\documentclass[12pt,reqno]{amsart}
\usepackage{cases}
\usepackage{amscd}
\usepackage{amsfonts}
\usepackage{amssymb}
\usepackage{amsmath}
\usepackage{graphicx}
\usepackage{epstopdf}
\usepackage{subfigure}
\usepackage{amsthm}
\usepackage{stmaryrd}

\theoremstyle{remark}
\newtheorem{example}{\textbf{Example}}[section]
\numberwithin{equation}{section}

\usepackage{color}
\usepackage{datetime}
\usepackage{tikz}

\makeatletter
\newcommand\figcaption{\def\@captype{figure}\caption}
\newcommand\tabcaption{\def\@captype{table}\caption}
\makeatother

\usepackage{geometry}
\usepackage{algorithm}
\usepackage{multirow}
\def\bq{\begin{equation}}
\def\eq{\end{equation}}
\def\bqs{\begin{equation*}}
\def\eqs{\end{equation*}}
\def\bsqs{\begin{subequations}}
\def\esqs{\end{subequations}}
\def\ba{\begin{aligned}}
\def\ea{\end{aligned}}
\def\br{\begin{eqnarray}}
\def\er{\end{eqnarray}}
\def\brr{\bq\begin{array}{rlll}}
\def\err{\end{array}\eq}

\def\text#1{\hbox{#1}}
\newtheorem{thm}{Theorem}[section]
\newtheorem{lem}[thm]{Lemma}
\newtheorem{coro}{Corollary}[section]
\newtheorem{rem}[thm]{Remark}
\newtheorem{alg}[thm]{Algorithm}

\newtheorem{definition}{Definition}[section]

\newcommand{\bsub}{\begin{subequations}}
\newcommand{\esub}{\end{subequations}$\!$}
\newcommand{\maK}{{\mathcal K}}

\title[FEMs for elliptic equations with line Dirac sources]{Regularity and finite element approximation for two-dimensional elliptic equations with line Dirac sources}

\author[H.~Li, X. ~Wan, P. Yin, L. Zhao]{Hengguang Li$^\dagger$, Xiang Wan$^\ddagger$, Peimeng Yin$^\dagger$ and Lewei Zhao$^\S$}
\address{$^\dagger$ Wayne State University, Department of Mathematics, Detroit, MI 48202} \email{li@wayne.edu; pyin@wayne.edu; Lewei.Zhao@wayne.edu}
\address{$^\ddagger$ George Washington University, Department of Mathematics, Washington, DC 20052} \email{xiangwan@gwu.edu}
\address{$^\S$ Beaumont Proton Therapy Center, Royal Oak, MI 48073} \email{Lewei.Zhao@beaumont.org}

\begin{document}

\begin{abstract} %Let $\Oemga\subset \mathbb{R}^2$ be a polygonal domain.
We study the elliptic equation with a line Dirac delta function as the source term subject to the Dirichlet boundary condition in a two-dimensional domain. Such a line Dirac measure causes different types of solution singularities in the neighborhood of the line fracture. We establish new regularity results for the solution in a class of weighted Sobolev spaces and propose finite element algorithms that approximate the singular solution at the optimal convergence rate. Numerical tests results are presented to justify the theoretical findings.
\end{abstract}

\maketitle

\bigskip

% \tableofcontents

%%%%%%%%%%%%%%%%%%%%%%%%%%%%%%%%%%%%%%%%%%%%%%%%%%%%%%%%%%%%

\section{Introduction}

Let $\Omega \subset \mathbb{R}^2$ be a polygonal domain and let  $\gamma$ be a line segment in $\Omega$.  Consider  the elliptic boundary value problem
\begin{equation}
\label{eq:Possion}
\left\{
\begin{aligned}
-\Delta u   =\delta_{\gamma}  &      \quad \text{in }  \Omega,    \\
u           =0  &       \quad \text{on }  \partial \Omega,
\end{aligned}
\right.
\end{equation}
where  the source term $\delta_\gamma$ is the line Dirac measure on $\gamma$, namely,
\[
\langle \delta_\gamma, v \rangle = \int_\gamma v(s)ds, \qquad \forall\ v \in L^2(\gamma).
\]
Such equations  occur in  many mathematical models including  monophasic flows in porous media, tissue perfusion or drug delivery by a network of blood vessels \cite{DAngelo12} and elliptic optimal control problems with controls acting on a lower dimensional manifold \cite{Gong14}.   Note that the line Dirac measure $\delta_\gamma$ is not an $L^2$ function.  Although  the solution tends to be  smooth in a large part of  the domain, it can become singular in the region close to the  one-dimensional (1D) fracture $\gamma$ and in the region close to the vertices of the domain, where the corner singularities are expected to rise. Since the corner singularity  associated to equation (\ref{eq:Possion}) is understood fairly well in the literature, we shall  address the concerns on the regularity of the solution near $\gamma$ and on the efficacy    of the numerical approximation.

%Our results can also be extended to applications with more general geometric descriptions on the curve $\gamma$ and allowing other functions in the source term.

%Here, $\gamma$ is a line segment  in  $\Omega$ such that for any point $x\in \gamma$ satisfying $\text{dist}{(x,\partial \Omega)}>0$.

Finite element approximations for second order elliptic equations with singular source terms have attracted considerable attention and many studies have focused on point singular measures. Babu\v{s}ka \cite{Babuska71},  Scott \cite{Scott73, Scott76}, and Casas \cite{Casas85} studied the convergence in the $L^2$ (or $H^\epsilon$ with small $\epsilon$) norm for Dirac measures centered at some points in 2D; and a review of the convergence rates can be found in \cite{Koppl14}, in which the authors considered the Dirac measures centered at some points in both 2D and 3D and showed that for $P_1$ finite elements quasi-optimal order and for higher order finite elements optimal order a priori estimates on a family of quasi-uniform meshes in $L^2$-norm on a subdomain excludes the locations of the delta source terms. For a Dirac measure centered at a point in a $N$-dimensional domain with $N\geq 2$, locally refined meshes around the singular point were used in \cite{Eriksson85} to improve the convergence rate. Graded meshes were used in \cite{Apel11} to study the convergence rate of the finite element approximation for a point Dirac measure in 2D and $L^2$ error estimate of order $h^2|\ln h |^\frac{3}{2}$ was obtained for approximations based on $P_1$ polynomials. More recently, 1D singular source terms have also attracted some attention. By assuming the regularity of an elliptic equation in 3D with a Dirac measure concentrated on a 1D fracture in a weighted Sobolev space, optimal finite element convergence rates were obtained in \cite{DAngelo08, DAngelo12} by using graded meshes. Then the authors in \cite{Ariche16} derived the 3D regularity for the simplified equation in \cite{DAngelo08, DAngelo12} when the Dirac measure concentrated on a line or segment fracture.

%Transmission problems for problems of the form (\ref{eq:Possion}) appear in many practical applications, especially when more than one type of material was involved and more details can be found in \cite{Li10} and reference therein.
%By studying the regularity of the transmission problem, the authors in \cite{Li10} obtained the regularity of a variety of strongly elliptic second order equation on a polygonal domain that may have cracks or vertices that touch the boundary.

In this paper, we  derive  regularity estimates  and propose optimal  finite element algorithms for equation (\ref{eq:Possion}). In particular, we investigate the solution regularity in a class of Kondratiev-type weighted spaces.  Note that the smoothness of the solution vary in different parts of the domain: the region close to the vertices, the neighborhood of the fracture $\gamma$, and the rest of the domain (Remark \ref{rk31}). By studying the local problem that inherits the line Dirac measure from equation (\ref{eq:Possion}), we obtain a ``full-regularity" estimate in these weighted spaces  in the neighborhood of $\gamma$. The key idea is to exploit the connection between the line Dirac measure and proper elliptic transmission problems in these weighted spaces. Based on the new regularity results and the existing regularity estimates on corner singularities, we in turn propose graded mesh refinement algorithms, such that the associated finite element methods of any order recover the optimal convergence rate in the energy norm even when the solution is singular. We study the model problem  (\ref{eq:Possion}) with a simple line fracture to simplify the exposition and avoid nonessential  complications in analysis. These results can be extended to more general cases, including the case where the single line fracture is replaced by multiple line fractures, whether intersecting or non-intersecting. With proper modifications, we also expect these analytical tools will be useful in the case when $\gamma$ is a smooth curve and when the source term $\delta_{\gamma}$ is replace by $q\delta_{\gamma}$ for  $q\in L^2({\gamma})$.

The rest of the paper is organized as follows. In Section \ref{sec-2}, we discuss the well-posedness and global regularity of  equation \eqref{eq:Possion} in Sobolev spaces. In Section \ref{sec-3}, we introduce the weighted spaces and derive the regularity estimates for the solution in the neighborhood of $\gamma$.  The main regularity results,  summarized in Theorem \ref{ureg}, imply that in addition to the lack of regularity in the direction across $\gamma$, the solution also possesses isotropic singularities at the endpoints of the line fracture. In Section \ref{sec-4}, we propose the finite element approximation of equation (\ref{eq:Possion}) based on a simple and explicit construction of graded meshes (Algorithm \ref{graded} and Remark \ref{rkgraded}). We further show that the proposed numerical methods achieve the optimal convergence rate by local interpolation error analysis in weighted spaces. We present various numerical test results in Section \ref{sec-5} to validate the theory.

Throughout the text below, we  denote by $ab$ the line segment with endpoints $a$ and $b$. The generic constant $C>0$ in our estimates  may be different at different occurrences. It will depend on the computational domain, but not on the functions  involved or the mesh level in the finite element algorithms.

%start with the regularity of the solution in classic Sobolev space and then we convert the problem into a transmission problem. Then, we study the regularity of the transmission problem and obtain the regularity for problem (\ref{eq:Possion}). In Section \ref{sec-4}, we obtained the quasi-optimal error estimate by using a sequence of graded triangular meshes. In Section \ref{sec-5}, numerical examples are given to verify the optimal convergence rate. Finally, concluding remarks are given in Section \ref{conrem}.

\section{Well-posedness and regularity in Sobolev spaces} \label{sec-2} %to Problem \eqref{eq:Possion}}
\subsection{Well-posedness of the solution}
Denote by $H^m(\Omega)$, $m\geq 0$,  the Sobolev space that consists of functions whose $i$th ($0\leq i\leq m$) derivatives are square integrable. Let $L^2(\Omega):=H^0(\Omega)$.
Denote by $H^1_0(\Omega)\subset H^1(\Omega)$  the subspace consisting of functions with  zero trace on the boundary  $\partial\Omega$. The variational formulation for equation (\ref{eq:Possion}) is
\begin{eqnarray}\label{eqn.weak}
a(u, v):=\int_\Omega\nabla u\cdot \nabla v dx=\langle \delta_\gamma, v \rangle, \quad \forall\ v\in H^1_0(\Omega).
\end{eqnarray}
According to the trace estimate \cite{LionsMaganesVolI}, $v|_\gamma$ is well defined in $L^2(\gamma)$ for $v\in H^1(\Omega)$. Therefore, it is clear that there exists a unique solution $u\in H^1_0(\Omega)$ defined by (\ref{eqn.weak}). However, the solution has limited regularity because the singular source term $\delta_\gamma\notin L^2(\Omega)$. In the rest of this section, we  present the global  regularity estimates for the solution in the domain.

\subsection{Regularity in  Sobolev spaces}
We begin with the regularity estimates of problem (\ref{eq:Possion}) in  Sobolev spaces $H^m$. We  first have the following  result regarding the line Dirac measure $\delta_\gamma$.
\begin{lem}
\label{lemma2-1}
Let $\Omega \subset \mathbb{R}^2$ be a bounded domain. Then
$\delta_\gamma \in H^{-\frac{1}{2}-\epsilon}(\Omega)$ for any  $\epsilon > 0$.
\end{lem}

\begin{proof}
The proof is based on the duality pairing (cf. \cite{LionsMaganesVolI}). Given  $\epsilon > 0$ and  $v \in H^{\frac{1}{2}+\epsilon}(\Omega)$,  by H\"older's inequality and the trace estimate \cite{Kesavan89, LionsMaganesVolI}, we have
\[
\langle \delta_\gamma, v \rangle = \int_\gamma v(s)ds \leq C \|v\|_{L^2(\gamma)} \leq  C \|v\|_{H^{\frac{1}{2}+\epsilon}(\Omega)}.
\]
Therefore, by the standard definition, we have
\[
\|\delta_{\gamma}\|_{H^{-\frac{1}{2}-\epsilon}(\Omega)} := \sup \{\langle \delta_\gamma, v \rangle \ : \  \|v\|_{H^{\frac{1}{2}+\epsilon}(\Omega)} = 1\} \leq C,
\]
which completes the proof.
\end{proof}
%\begin{rem}
%A more general result about the regularity of $\delta_{\gamma}$ with other cases of dimensions of the domain and of the manifold $\gamma$ can be found in \cite[Theorem 2.1]{Gong14}. In particular, in their result, one chooses $k=1, n=2$, and the following conclusion is consistent with Lemma \ref{lemma2-1}.
%\end{rem}

Consequently, we have the following global regularity estimate for the solution.
\begin{lem}
\label{thm2-2}
Given  $\epsilon>0$, the solution of equation \eqref{eq:Possion} satisfies $u \in H^{\frac{3}{2}-\epsilon}(\Omega) \cap H^1_0(\Omega)$.
\end{lem}

\begin{proof}
From Lemma \ref{lemma2-1}, it follows $\delta_\gamma \in H^{-\frac{1}{2}-\epsilon}(\Omega)$. Then the standard elliptic regularity theory \cite{Alinhac07} leads to the conclusion.
\end{proof}

%\begin{proof}
%Define the positive self-adjoint operator on $L^2(\Omega)$:
%\begin{equation}
%\label{E2-2}
%Af  = -\Delta f, \qquad
%\mathcal{D}(A) = H^2(\Omega) \cap H^1_0(\Omega),
%\end{equation}
%so that problem \eqref{eq:Possion} can be rewritten abstractly as
%\[
%Au=q\delta_{\gamma}.
%\]
%
%It is well known that the Dirichlet Laplacian $-A$ maps $H_0^1(\Omega)$ to $H^{-1}(\Omega)$ (\cite[Page 322]{Evans15}) and maps $\mathcal{D}(A) = H^2(\Omega) \cap H^1_0(\Omega)$ to $L^2(\Omega)$. Hence, by Lemma \ref{lemma2-1} and interpolation of Sobolev space (\cite[Chapter 14]{Brenner11}), $u \in [H^2(\Omega), H^1_0(\Omega)]_{\theta}$ with $\theta = \frac{1}{2} + \epsilon$.
%\end{proof}

%\begin{rem}
%One significance of Lemma \ref{thm2-2} is the boundedness of $u$ in the 2D case. Indeed, by the Sobolev Embedding Theorem, $H^{\frac{3}{2}-\epsilon}(\Omega)$ can be embedded into $L^\infty(\Omega)$ for any bounded set $\Omega \subset \mathbb{R}^2$. %This result holds true in the $\mathbb{R}^3$ case as proven in \cite{Ariche16}.
%\end{rem}

Thus, by  Lemma \ref{thm2-2} and the Sobolev embedding theorem \cite{Ciarlet74}, we obtain
\begin{coro}\label{co1}
The solution $u$ of equation \eqref{eq:Possion} is H\"{o}lder continuous $u\in C^{0,1/2-\epsilon}(\Omega)$ for any small $\epsilon>0$. In particular, we have  $u\in C^0({\Omega})$.
\end{coro}

Based on Lemma \ref{thm2-2} and Corollary \ref{co1}, the solution is merely in $H^{\frac{3}{2}-\epsilon}(\Omega)$ for $\epsilon>0$. The lack of regularity is largely due to the singular line Dirac measure $\delta_\gamma$ in the source term. However, regularity is a local property. Such solution  singularity shall occur only in the neighborhood of $\gamma$. In a large part of the domain, the solution is reasonably smooth. Hence, we shall study the regularity of equation (\ref{eq:Possion}) in some weighted Sobolev spaces that can accurately characterize the local behavior of the solution.

%Because of lack of regularity, the finite element approximations for the problem \eqref{eq:Possion} on a sequence of quasi-uniform triangular meshes in the domain $\Omega$ dose not give optimal rates of convergence. To obtain a better approximation, we consider the problem (\ref{eq:Possion}) in weighted Sobolev space (see e.g., \cite{Li10}) and then construct a sequence of (graded) triangular meshes.

\section{Regularity estimates in weighted spaces}
\label{sec-3}
Recall the domain $\Omega$ and the line segment $\gamma$ in equation (\ref{eq:Possion}). Without loss of generality, we assume  $\gamma=\{(x,0), \ 0< x<1\}$ with the endpoints $Q_1=(0,0)$ and  $Q_2=(1,0)$ as shown in Figure \ref{fig:Omega}. Let  $\mathcal{V}$ be the singular set, which is the collection of $Q_1$, $Q_2$, and all the vertices of $\Omega$. In this section, we first study an auxiliary transmission problem in Subsections \ref{31} and \ref{32}. Then, we  obtain the regularity estimates for equation (\ref{eq:Possion}) in Subsection \ref{33}.
\subsection{The transmission  problem}\label{31}
%Transmission problem appears from many real world applications, particularly for problems involving more than one type of materials. Motivated by the work in  \cite{Li10}, where analysis of the finite element methods for transmission problems was studied, we construct an appropriate transmission problem for  the problem (\ref{eq:Possion}).

Consider the equation
\begin{equation}
\label{eq:2d}
\left\{
\begin{aligned}
-\Delta w=0 &\quad \mbox{ in   }\Omega \setminus \gamma,\\
w_y^+=w_y^--1 &\quad \mbox{ on   } \gamma,\\
w^+=w^- &\quad \mbox{ on   } \gamma,\\
w=0 &\quad \mbox{ on   } \partial \Omega,
\end{aligned}
\right.
\end{equation}
where  $w_y=\partial_yw$.
Here, for a function $v$,
$
v^\pm:=\lim_{\epsilon\rightarrow 0}v(x,y\pm\epsilon).
$
It is clear that equation (\ref{eq:2d}) has a unique weak solution $$w\in H^1(\Omega\setminus \gamma)\cap \{w|_{\partial\Omega}=0\}.$$

\begin{rem}\label{rk31}We define different regions of the domain  as follows for further local regularity estimates.
Denote by $\mathbb{H}^+$ and $\mathbb{H}^-$ the upper and lower half planes, respectively. Define $\gamma_0=\{(x,0): d \leq x \leq 1-d\} \subset \gamma$ for some small $d>0$. Then we choose two open subsets $\Omega^+\subset \Omega \cap {\mathbb{H}^+}$ and $\Omega^-\subset \Omega \cap {\mathbb{H}^-}$, each of whom has a  smooth boundary and is away from $\partial\Omega$,  such that
$\gamma_0 = \overline{\Omega^+}\cap \overline{\Omega^-}$. Let $B(x_0, r)$ be the ball centered at $x_0$ with radius $r$.
Denote by $B_i=B(Q_i,2d)$, $i=1, 2$, the neighborhoods around the endpoints of $\gamma$.
 See Figure \ref{fig:decom_L_R}. We assume $d$ is sufficiently small such that $B_1\cap B_2=\emptyset$ and  $(B_1\cup B_2)\cap \partial \Omega=\emptyset$.  Therefore,  the domain $\Omega$ is divided into three  regions: (i) the interior region $R_1={\Omega^+} \cup {\Omega^-}$ away from the set $\mathcal V$, (ii) the region $R_2=B_1\cup B_2$ consisting of the neighborhoods of the endpoints of $\gamma$, and (iii) $R_3=\Omega\setminus (\bar R_1\cup \bar R_2)$ is the region close to the boundary $\partial \Omega$.
 \end{rem}
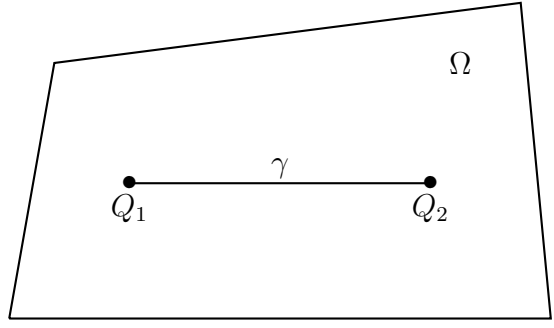
\begin{figure}
\begin{center}
\begin{tikzpicture}[scale=0.2]
%\draw[thick]
%(-15,5)
%.. controls (-20,-9) and (-5,-18) .. (15,-7)
%.. controls  (24,-1) and (15,13).. (6,8)
%.. controls (3,6) and (-3,3) .. (-6,6)
%.. controls (-9,10) and (-13,10) .. (-15,5);
%inside blob
\draw[thick]
(-18,-11) -- (-15,6) -- (16,10) -- (18,-11) -- (-18,-11);
%inside blob
\draw[thick] (-10,-2) node {$\bullet$} node[anchor = north] {$Q_1$} --
(10,-2) node {$\bullet$} node[anchor = north] {$Q_2$};

\draw (12,6) node {$\Omega$};
\draw(0,-1) node{$\gamma$};
\end{tikzpicture}
\end{center}
\vspace*{-15pt}
    \caption{Domain $\Omega$ containing a line fracture $\gamma$.}
    \label{fig:Omega}
\end{figure}

%\subsubsection{Regularity of transmission problem}

%The regularity of the solution to transmission problem \eqref{eq:2d} on the singular part $\Omega_I$ turns out to be different around the end points of $\gamma$ and away from them (see Lemma \ref{lemma3-4} for the precise results). To better describe it, we take a closer look at sub-regions of $\Omega$ as follows.

%As noticed above, $(\Omega^+ \cup \Omega^-) \backslash \gamma_0$ is the `good' part, where the solution $w$ to problem \eqref{eq:2d} is relatively more regular by the elliptic theory compared with on $\Omega_L\cup \Omega_R$ - the singular part.

\begin{rem}\label{r3} In region $R_3$, the solution regularity in (\ref{eq:2d}) is determined by the geometry of the domain. In particular, the solution can possess singularities near the non-smooth points (vertices) of the boundary. The regularity estimates in this region is well understood in the literature. See for example \cite{Apel99, Dauge88, Grisvard85, Kondratiev67, Li10} and references therein. Therefore, we shall concentrate on the regularity analysis in regions $R_1$ and $R_2$ for equation (\ref{eq:2d}).
\end{rem}

\begin{figure}
\begin{center}
\begin{tikzpicture}[scale=0.22]
%\draw[thick]
%(-15,5)
%.. controls (-20,-9) and (-5,-18) .. (15,-7)
%.. controls  (24,-1) and (15,13).. (6,8)
%.. controls (3,6) and (-3,3) .. (-6,6)
%.. controls (-9,10) and (-13,10) .. (-15,5);

% \fill[red!15] (0,-2) ellipse (12 and 3);

\draw[thick]
(-18,-11) -- (-15,6) -- (16,10) -- (18,-11) -- (-18,-11);

%inside blob
\draw (-10,-2) -- (10,-2);

\filldraw[color=red!60, fill=red!15] (-8,-2)
.. controls (-10,-3) and (-10,-5) .. (-8,-6)
.. controls (-4, -8) and (4,-8) .. (8,-6)
.. controls (10, -5) and (10,-3) .. (8,-2);

\filldraw[color=red!60, fill=red!15] (-8,-2)
.. controls (-10,-1) and (-10,1) .. (-8,2)
.. controls (-4, 4) and (4,4) .. (8,2)
.. controls (10, 1) and (10,-1) .. (8,-2) ;

\draw[thick]
(-8,-2) node {$\boldsymbol{\cdot}$}  node[anchor = north] {$d$} --
(8,-2) node {$\boldsymbol{\cdot}$} node [anchor = north] {$1-d$};

\draw[ultra thick]  (-10,-2) node {$\boldsymbol{\cdot}$};
\draw[ultra thick]  (10,-2) node {$\boldsymbol{\cdot}$};

\draw (-11.5,-2.5) node {$Q_1$};
\draw (11.5,-2.5) node {$Q_2$};

\draw (-10,-2) circle (4);
\draw (10,-2) circle (4);

\draw (12,6) node {$\Omega$};
\draw (0,2) node {$\Omega^+$};
\draw (0,-6) node {$\Omega^-$};
\draw(0,-3) node{$\gamma_0$};
\draw(-11,1) node{$B_1$};
\draw(11,1) node{$B_2$};
\end{tikzpicture}
\end{center}
\vspace*{-15pt}
    \caption{Decomposition around the singular line: $\Omega^+, \Omega^-, B_1$ and $B_2$.}
    \label{fig:decom_L_R}
\end{figure}
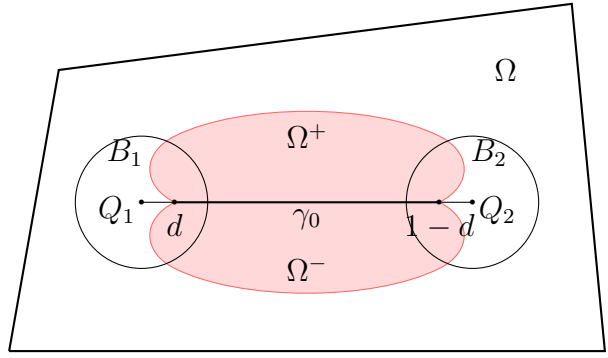

We now  introduce a class of  Kondratiev-type weighted spaces for the analysis of equation (\ref{eq:2d}).
\begin{definition} \label{wss} (Weighted Sobolev spaces) Recall the set $\mathcal V$ that consists of the endpoints of $\gamma$ and all the vertices of the domain $\Omega$.
Let $r_i(x,Q_i)$ be the distance from $x$ to $Q_i \in \mathcal{V}$ and let
\begin{eqnarray}\label{eqn.rho}
\rho(x)=\Pi_{Q_i\in \mathcal{V}} r_i(x,Q_i).
\end{eqnarray}
For $a\in\mathbb R$, $m\geq 0$, and $G\subset \Omega$,  we define the weighted Sobolev space
$$
\maK_{a}^m(G) := \{v,\ \rho^{|\alpha|-a}\partial^\alpha v\in L^2(G), \forall\ |\alpha|\leq m \},
$$
where the multi-index $\alpha=(\alpha_1,\alpha_2)\in\mathbb Z^2_{\geq 0}$, $|\alpha|=\alpha_1+\alpha_2$, and $\partial^\alpha=\partial_x^{\alpha_1}\partial_y^{\alpha_2}$.
The $\maK_{a}^m(G)$ norm for $v$  is defined by
$$
\|v\|_{\maK_{a}^m(G)}=\big(\sum_{|\alpha|\leq m}\iint_{G} |\rho^{|\alpha|-a}\partial^\alpha v|^2dxdy\big)^{\frac{1}{2}}.
$$
\end{definition}
\begin{rem}
According to Definition \ref{wss}, in the region that is away from the set $\mathcal V$, the weighted space $\maK^m_a$ is equivalent to the Sobolev space $H^m$. In the region $R_3$ (see Remark \ref{rk31}) that is close to the vertices of the domain, the space $\maK^m_a$ is the same Kondratiev space for analyzing corner singularities \cite{Dauge88,Grisvard85,Kondratiev67}. In contrast to the Kondratiev space where the weight is the distance function to the vertex set, the weight in the space $\maK^m_a$ also consists of the distance function to the endpoints of $\gamma$. In particular, for $i=1,2$, in the neighborhood $B_i$ (Figure \ref{fig:decom_L_R}) of an endpoint $Q_i$ of $\gamma$, the weighted space can be written as
\begin{equation*} \label{def_weighted_Sobolev}
\maK_a^m(B_i) =
\{v, r_i^{|\alpha|-a}\partial^\alpha v\in L^2(B_i), \forall\ |\alpha|\leq m \}.
\end{equation*}
\end{rem}

In each $B_i$, we further define $\chi_i \in C_0^\infty(B_i)$ that satisfies
\begin{equation*}
\label{WS}
\chi_i=\left\{
\begin{aligned}
1&      \quad \mbox{ in $B(Q_i,d)$},    \\
0&     \quad \mbox{ on }  \partial B_i.
\end{aligned}
\right.
\end{equation*}
Note that supp$(\chi_1)\cap$supp$(\chi_2)=\emptyset$. In addition, we denote by
\begin{eqnarray}\label{w}W= {\rm{span}}\{\chi_i\}, \quad i=1, 2, \end{eqnarray} the linear span of these two functions.
% that are constants equal to $1$ in the neighborhood of $Q_i$ and satisfy $\chi_i=0$ on $\partial \Omega_i$. These functions have disjoint supports.

%See more in the proof of Lemma \ref{lemma3-4}, Eq. \eqref{E14}.

%The specific regularity results on these two parts are presented in the following technical lemma:

%\begin{lem}
%\label{lemma3-4}
%Let $w$ be the solution to problem \eqref{eq:2d}, then $w$ is smooth away from the singular line, that is $w\in H^m((\Omega^+\cup \Omega^-)\backslash \gamma_0)$. Moreover, there is a decomposition of $w$ near the endpoints $Q$ of the singular line, i.e, $w=z_{reg}+v+\chi_Q$, where $z_{reg}\in K_\alpha^m(\Omega_L)$,$v\in K_\beta^m(\Omega_L)$ for $\alpha<\frac{1}{2}$, $\beta<2$ and any $0\leq m\leq5$, $\chi_Q\in C^\infty(\Omega_L)$ is constant equal to 1 in a neighborhood of Q and satisfies $\frac{\partial \chi_Q}{\partial y}=0$ on $\gamma \cap \Omega_L$. (Similarly on $\Omega_R$)
%\end{lem}
\subsection{Regularity estimates for equation (\ref{eq:2d})}\label{32} We now proceed to carry out the regularity analysis for the transmission problem  (\ref{eq:2d}). Recall the interior region $R_1={\Omega^+} \cup {\Omega^-}$ in Remark \ref{rk31}. We start with the regularity analysis for the solution  in $R_1$.
\begin{lem}\label{waway}
The solution of equation   \eqref{eq:2d} is smooth in either ${\Omega^+}$ or in  ${\Omega^-}$. Namely, for any $m\geq 1$, $w\in H^{m+1}({\Omega^+})$ and  $w\in H^{m+1}({\Omega^-})$.
\end{lem}
\begin{proof}
Recall that ${\Omega^+}$ and ${\Omega^-}$ are regions with a smooth boundary. Therefore, by the trace estimate, for $m\geq 1$, we can find two functions  $w_U\in H^{m+1}(\Omega^+)$ and $w_D\in H^{m+1}(\Omega^-)$ such that $w_U=w_D$ and $\frac{\partial w_U}{\partial y}=\frac{\partial w_D}{\partial y}-1$ on $\gamma_0:=\overline{{\Omega^+}} \cup \overline{{\Omega^-}}\subset\gamma$.

Define
\begin{equation}
w_0=
\left\{
\begin{aligned}
w_U \quad & \mbox{ in   } \Omega^+, \\
w_D \quad & \mbox{ in   } \Omega^-.
\end{aligned}
\right.
\end{equation}
Then $w-w_0$ satisfies the standard transmission problem with a smooth interface
\begin{equation}
\left\{\begin{aligned}
-\Delta(w-w_0)=\Delta w_0& \quad \mbox{ in   } (\Omega^+\cup \Omega^-),   \\
(w-w_0)_y^+=(w-w_0)_y^- &\quad \mbox{ on   } \gamma_0,\\
(w-w_0)^+=(w-w_0)^- & \quad \mbox{ on   } \gamma_0.
\end{aligned}\right.
\end{equation}
Therefore, by the regularity results in \cite{Nicaise92, Li10}, we have $w-w_0\in H^{m+1}(\Omega^+)$ and $w-w_0\in H^{m+1}(\Omega^-)$, which leads to the desired result.
\end{proof}

We now concentrate on the solution behavior in the neighborhood $B_i$, $i=1,2$, of an endpoint of $\gamma$ (see Remark \ref{rk31}).  We first consider the following  problem  with a simpler transmission  condition on $\gamma$,
\begin{equation}\label{Lisimiliar}
\left\{
\begin{aligned}
-\Delta z =f & \quad \text{in } B_i\setminus \gamma, \\
{z}_y^+={z}_y^- & \quad \text{on }  \gamma \cap B_i, \\
{z}^+={z}^- & \quad \text{on }  \gamma \cap B_i, \\
z=0 & \quad \text{on } \partial B_i.
\end{aligned}
\right.
\end{equation}
We recall a regularity result  in  \cite{Li10}   regarding $z$ in the neighborhood of $Q_i$.

\begin{lem}\label{singular} For equation (\ref{Lisimiliar}), there exists $b_{Q_i}>0$ such that the following statement holds.
Let $0<a<b_{Q_i}$ and $m \geq 1$. Assume $f \in \maK_{a-1}^{m-1}(B_i\setminus \gamma)$. Recall the finite dimensional space $W$ in (\ref{w}). Then, there exists a unique decomposition $z=z_{reg}+z_s$,
such that $z_{reg} \in  \maK_{a+1}^{m+1}(B(Q_i, d)\setminus\gamma)$ and $z_s \in W$.
Moreover, it follows
\bq
\|z_{reg}\|_{\maK_{a+1}^{m+1}(B(Q_i, d)\setminus\gamma)}+\|z_s\|_{L^\infty(B_i)} \leq C \|f\|_{\maK_{a-1}^{m-1}(B_i\setminus\gamma)},
\eq
where the constant $C>0$ is independent of $f$.
\end{lem}
\begin{rem}Based on the calculation in \cite{Li10},  the constant $b_{Q_i}$  is determined by the smallest positive eigenvalue of  the operator $-\partial_\theta^2$ in $(0,2\pi)$ with the periodic boundary condition. Note that  $k^2$,  $k\in \mathbb{Z}_{\geq 0}$, are these eigenvalues. Thus, it follows $b_{Q_i}=1$.
\end{rem}

Recall the solution $w$ of the transmission problem (\ref{eq:2d}). Recall the space $W$ in (\ref{w}).  Then,   in the neighborhood $B_i$ of $Q_i$, $i=1,2$, we have the following regularity result.

%we use a local polar coordinate system $(r, \theta)$ such that $Q_i$ is the origin.  Then, we have the following lemma for functions in the polar coordinates.
%\begin{prop}\label{wg}
%In polar coordinate system, if a function $g(\theta) \in H^{m+1}([0,2\pi])$ for $m\geq 1$, then there exists a function $w_g \in K_{a+1}^{m+1}(B_1)$ for $\forall a$ and satisfies $w_g(2d, \theta)=g(\theta)$.
%\end{prop}
%\begin{proof}
%We set $w_g(\rho, \theta)=\eta g(\theta)$, where $\eta \in C^\infty(B_1)$ satisfying $\eta=0$ for $\rho \leq d$ and $\eta=1$ for $\rho=2d$.  Since $\eta \in C^\infty(B_1)$, $\eta=0$ for $\rho \leq d$ and $g(\theta) \in H^{m+1}([0,2\pi])$, then $w_g \in K_{a+1}^{m+1}(B_1)$. Due to $\eta=1$ for $\rho=2d$, thus it follows $w_g(2d, \theta)=g(\theta)$.
%and $g_1(\rho,\theta)=w(2\delta,\theta)-2q\delta \sin\frac{\theta}{2}$, then for $m \geq 1$ and $a<1$, it follows
%\bq
%\ba
%& w_g \in K_{a+1}^{m+1}(B_1),\\
%& (w_g)_{\rho\rho}+\frac{1}{\rho} (w_g)_\rho+\frac{1}{\rho^2} (w_g)_{\theta\theta} \in K_{a-1}^{m-1}(B_1),
%\ea
%\eq
%\end{proof}
%By Proposition \ref{wg}, we have for $\forall a$,
%\bq\label{lapwg}
%(w_g)_{\rho\rho}+\frac{1}{\rho} (w_g)_\rho+\frac{1}{\rho^2} (w_g)_{\theta\theta} \in K_{a-1}^{m-1}(B_1).
%\eq

\begin{thm}\label{lemma3-4}
Let $B_{d, i}:=B(Q_i, d)\subset B_i$, $i=1, 2$. Then, in $B_{d,i}$,  the solution $w$ of equation \eqref{eq:2d} admits a decomposition
$$
w=w_{reg}+w_s,
$$
where $w_s \in W$ and $w_{reg} \in \maK_{a+1}^{m+1}(B_{d,i}\setminus\gamma)$ for $0<a<1$ and $m\geq 1$. Moreover, we have
\bq\label{regularity}
\|w_{reg}\|_{\maK_{a+1}^{m+1}(B_{d, i}\setminus\gamma)}+\|w_s\|_{L^\infty(B_{i})} \leq C.
\eq
%where $v=-r \sin \frac{\theta}{2}$  and $w_g \in K_{a+1}^{m+1}(B_1)$ and $f=-\Delta (v+w_g) \in K_{a-1}^{m-1}(B_1)$.
\end{thm}

\begin{proof} We shall derive the theorem in $B_{d,1}$. The proof in $B_{d,2}$ can be carried out in a similar manner. Let ($r, \theta$) be the local polar coordinates  in $B_1$ for which $Q_1$ is at the origin and $\theta=0$ corresponds to the positive $x$-axis.  We shall use a localization argument to obtain the estimate. In the rest of the proof, we simplify the notation for $B_{d,1}$ by letting $B_d=B_{d, 1}$.

{\bf Step 1.} Let $\eta\in C^\infty_0(B_1)$ be a cutoff function such that $\eta=1$ in $B_d$, $\eta=0$ for $r>3d/2$, and $\eta_\theta:=\partial_\theta\eta=0$. Define $q:=\eta w$. Note that on $\gamma$ ($\theta=0, 2\pi$), we have
\begin{eqnarray*}
q_y^+&=&(\sin\theta)^{+} q_r^+ +\frac{(\cos\theta)^+}{r}q_\theta^+=\frac{1}{r}q_\theta^+\\
&=&\frac{1}{r}\eta w_\theta^+=\eta\left((\sin\theta)^{+} w_r^+ +\frac{(\cos\theta)^+}{r}w_\theta^+\right)=\eta w_y^+,
\end{eqnarray*}
where for a function $v(r, \theta)$,
$
v^\pm:=\lim_{\epsilon\rightarrow 0}v(r,\theta\pm\epsilon).
$
With a similar calculation, we have $q_y^-=\eta w_y^-$ on $\gamma$.
Then,  according to the transmission condition in equation (\ref{eq:2d}), we have
$$
q_y^+=\eta w^+_y=\eta(w^-_y-1)=q_y^--\eta, \qquad {\rm{on}}\ \gamma.
$$
%$v_y^+=\sin(\theta)^{+} v_r^+ +\frac{\cos(\theta)^+}{r}v_\theta^+=  \sin(\theta)^- v_r^- +\frac{\cos(\theta)^-}{r} v_\theta^- - \eta=v_y^--\eta & && \mbox{ on   } \gamma,$
Consequently, $q$ satisfies the following equation
%\iffalse
%\begin{equation}
%\label{eq:polaro}
%\left\{\begin{aligned}
%(\rho\frac{\partial}{\partial \rho})^2w+ \frac{\partial^2}{\partial \theta^2}w&=0 && \mbox{ in   } \Omega_L\\
%\sin \theta\frac{\partial w}{\partial \rho}+\frac{\cos \theta}{\rho}\frac{\partial w}{\partial \theta}&=1 && \mbox{ on   } \theta=0^+\\
%\sin \theta \frac{\partial w}{\partial \rho}+\frac{\cos \theta}{\rho}\frac{\partial w}{\partial \theta}&=0 && \mbox{ on   }\theta=0^-\\
%w&=w|_{\partial \Omega_L}&&\mbox{ on }\rho=2\epsilon
%\end{aligned}\right.
%\end{equation}
%
%\begin{equation}
%\label{eq:polaro}
%\left\{\begin{aligned}
%w_{\rho\rho}+\frac{1}{\rho} w_\rho+\frac{1}{\rho^2} w_{\theta\theta}=0 & && \mbox{ in   } \Omega_L,\\
%\sin \theta w_\rho+\frac{\cos \theta}{\rho}w_\theta=1 & && \mbox{ on   } \theta=0^+,\\
%\sin \theta w_\rho+\frac{\cos \theta}{\rho}w_\theta=0 &  && \mbox{ on   }\theta=2\pi^-,\\
%w=w|_{\partial \Omega_L} & &&\mbox{ on }\rho=2\epsilon,
%\end{aligned}\right.
%\end{equation}
%\fi

\begin{equation}
\label{eq:polaro}
\left\{\begin{aligned}
-\Delta q=-\Delta(w\eta) & && \mbox{ in   } B_1\setminus\gamma,\\
q_y^+=q_y^--\eta & && \mbox{ on   } \gamma,\\
q^+=q^- &  && \mbox{ on   }\gamma,\\
q=0 & &&\mbox{ on }\partial B_1.
\end{aligned}\right.
\end{equation}
Note that based on the definition of $\eta$, in $B_1\setminus \gamma$, $-\Delta(w\eta)=-2\nabla w\cdot\nabla \eta-w\Delta \eta$ and in $B_d\setminus \gamma$, $-\Delta(w\eta)=0$.
%Here, we assume $\theta=0$ corresponds to the positive $x$-axis and for a function $\phi(r, \theta)$

%The last condition in (\ref{eq:polaro}) for the value of $w$ on $\partial B_1$ is defined by the continuity of $w$.

{\bf Step 2.}
Define $p(r,\theta)=-\eta r \sin \frac{\theta}{2}$ for $0\leq\theta\leq2\pi$, where $\eta$ is defined in Step 1.
Then $p\in H^1(B_1)$ satisfies
\begin{equation}
\label{eq:polarv}
\left\{\begin{aligned}
-\Delta p= \Delta\left(\eta r\sin\frac{\theta}{2}\right) &  && \mbox{ in   } B_1\setminus\gamma,\\
p_y^+=(\sin\theta)^+ p_r^+ +\frac{(\cos\theta)^+}{r} p_\theta^+=-\frac{1}{2}\eta  &  && \mbox{ on   }  \theta=0,\\
p_y^-=(\sin\theta)^- p_r^- +\frac{(\cos\theta)^-}{r} p_\theta^-=\frac{1}{2}\eta  &  && \mbox{ on   } \theta=2\pi,\\
%v=0 &  && \mbox{ on   }  \gamma,\\
p=0 & && \mbox{ on } \partial B_1.
\end{aligned}\right.
\end{equation}
%where
%$$
%f_1(r,\theta)=-\frac{3}{4r}\sin{\frac{\theta}{2}}.
%$$
It is worth noting that $p \not \in H^2(B_1)$. However, by a straightforward calculation, it is clear that $p\in \maK_{a+1}^{m+1}(B_1)$ and $\Delta(\eta r\sin\frac{\theta}{2}) \in \maK_{a-1}^{m-1}(B_1)$ for any $m\geq1$ and $0<a<1$.

{\bf Step 3.} Let $z=p-q$. Then, based on equations (\ref{eq:2d}), (\ref{eq:polaro}), and (\ref{eq:polarv}), $z$ satisfies \begin{equation}
\label{eq:zz}
\left\{\begin{aligned}
    -\Delta z= f & \quad \text{in }  B_1\setminus \gamma, \\
    z_y^+=z_y^- & \quad \text{on }  \gamma , \\
    z^+=z^- & \quad \text{on }  \gamma , \\
    z=0& \quad \text{on } \partial B_1,
\end{aligned}\right.
\end{equation}
where $f=-\Delta(w\eta)-\Delta(\eta r\sin\frac{\theta}{2})$. Note that  by the fact $\Delta(w\eta)=0$ in $B_d\setminus \gamma$ and by  Lemma \ref{waway}, $f\in \maK_{a-1}^{m-1}(B_1)$ for any $m\geq1$ and $0<a<1$.
Applying Lemma \ref{singular} to equation (\ref{eq:zz}), we conclude that there exists a unique decomposition $z=z_{reg}+z_s$, with $z_{reg} \in  \maK_{a+1}^{m+1}(B_d\setminus\gamma)$ and $z_s \in W$, satisfying
\bq\label{eq.zzz}
\|z_{reg}\|_{\maK_{a+1}^{m+1}(B_d\setminus\gamma)}+\|z_s\|_{L^\infty(B_1)} \leq C \|f\|_{\maK_{a-1}^{m-1}(B_1\setminus\gamma)}.
\eq
Since $\eta w=q=p-z$, by the estimate (\ref{eq.zzz}) and by the definition of $p$ in Step 2, we obtain the decomposition of $w$ in $B_d\setminus \gamma$:
$$
w=w_{reg}+w_s,
$$
where $w_{reg}=p-z_{reg}$ and $w_s=-z_s$, such that for any $m\geq1$ and $0<a<1$,
\bqs
\|w_{reg}\|_{\maK_{a+1}^{m+1}(B_d\setminus\gamma)}+\|w_s\|_{L^\infty(B_1)} \leq C \|f\|_{\maK_{a-1}^{m-1}(B_1)}+\|p\|_{\maK_{a+1}^{m+1}(B_d\setminus\gamma)}<C,
\eqs
which completes the proof.

%\begin{center}
%\begin{tikzpicture}
%  \draw
%    (3,-1) coordinate (a) node[right] {}
%    -- (0,0) coordinate (b) node[left] {}
%    -- (2,2) coordinate (c) node[above right] {}
%    pic["$\alpha_Q$", draw=orange, <->, angle eccentricity=1.2, angle radius=1cm]
%    {angle=a--b--c}
%
%   (-3,0)--(-1,0);
%
%   \draw
%   [color=orange,<->](-2,0) arc (2:358:1cm);
%
%   \draw (-3,0.2) node {$\alpha_Q=2\pi$}
%         (-1.5,-0.2) node{$\gamma$};
%\end{tikzpicture}
%\end{center}

\end{proof}

%Similarly result can be achieved on the small neighborhood $\Omega_2$ of the right endpoint.

\subsection{Regularity estimates for equation (\ref{eq:Possion})}\label{33}
 Recall that $\mathcal V$ consists of the endpoints of $\gamma$ and all the vertices of $\Omega$. Recall $B_{d, i}:=B(Q_i, d)$ in Theorem \ref{lemma3-4}, and the regions $\Omega^+$, $\Omega^-$, $R_3$ in Remark \ref{rk31}. We are now ready to derive the regularity estimate for the solution of equation \eqref{eq:Possion} with the line Dirac measure.

 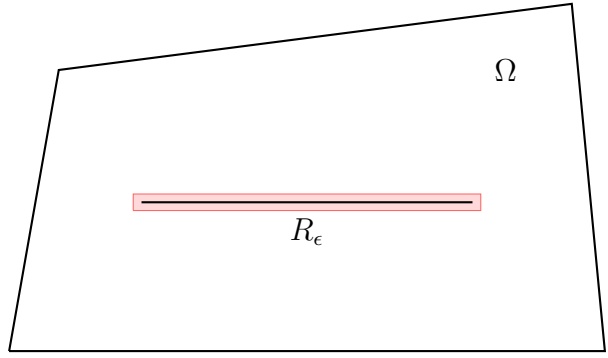
\begin{figure}
\begin{center}
\begin{tikzpicture}[scale=0.22]
%\draw[thick]
%(-15,5)
%.. controls (-20,-9) and (-5,-18) .. (15,-7)
%.. controls  (24,-1) and (15,13).. (6,8)
%.. controls (3,6) and (-3,3) .. (-6,6)
%.. controls (-9,10) and (-13,10) .. (-15,5);

\draw[thick]
(-18,-11) -- (-15,6) -- (16,10) -- (18,-11) -- (-18,-11);

\filldraw[color=red!60, fill=red!15]
(-10.5,-2.5) -- (-10.5,-1.5) -- (10.5,-1.5) -- (10.5,-2.5) -- (-10.5,-2.5);

\draw[thick] (-10,-2) -- (10,-2);

\draw (0,-3.8) node {$R_{\epsilon}$};

\draw (12,6) node {$\Omega$};
\end{tikzpicture}
\end{center}
\vspace*{-15pt}
    \caption{A small neighborhood $R_\epsilon$ of the line fracture $\gamma$.}
    \label{fig:R_epsilon}
\end{figure}

\begin{thm}\label{ureg}
The solution $u$ of equation (\ref{eq:Possion}) is smooth in the region away from the set $\mathcal V$, namely, for  $m\geq 1$, $u\in H^{m+1}(\Omega^+)$ and $u\in H^{m+1}(\Omega^-)$. In the neighborhood of each endpoint of $\gamma$,
$u$ admits a decompostion
$$
u=u_{{reg}}+u_s, \qquad u_s\in W,
$$
such that for any $m\geq 1$ and $0<a<1$,
\bqs
\|u_{reg}\|_{\maK_{a+1}^{m+1}(B_{d,i}\setminus\gamma)}+\|u_s\|_{L^\infty(B_i)} \leq C.
\eqs
In the region $R_3$ away from $\gamma$ and close to the boundary, $u\in\maK^{m+1}_{a+1}(R_3)$ for $m\geq 1$ and $0<a<\frac{\pi}{\omega}$, where $\omega$ is the largest interior angle among all the vertices of the domain $\Omega$.
\end{thm}
\begin{proof} Recall the solution $w$ of the transmission problem (\ref{eq:2d}). We shall show $u=w$. We first extend $w$ to $\Omega$ by defining
\begin{equation}\label{u_alter}
w:=
\left\{
\begin{aligned}
w\quad  & \text{in } \Omega\backslash \gamma, \\
w^+(=w^-)\quad  & \text{on }  \gamma.
\end{aligned}
\right.
\end{equation}
For  $\epsilon>0$ small, define $R_{\epsilon}:=\{(-\epsilon, 1+\epsilon) \times (-\epsilon, \epsilon)\}$ to be a small neighborhood of $\gamma$.  Let $\mathbf{n}_\epsilon$ be the unit outward normal vector to $\partial R_{\epsilon}$. See Figure \ref{fig:R_epsilon}. Let $\tilde{u}=u-w$. Then for any  $\phi \in C_0^\infty(\Omega)$, it follows
\bq
\ba\label{uwdiff}
-\iint_\Omega \Delta \tilde{u} \phi dxdy = & -\iint_\Omega \Delta u \phi dxdy +\iint_\Omega \Delta w \phi dxdy \\
= & \iint_\Omega \delta_\gamma \phi dxdy + \iint_{\Omega\backslash R_{\epsilon}} \Delta w \phi dxdy+ \iint_{R_{\epsilon}} \Delta w \phi dxdy \\
= & \int_\gamma \phi ds + \iint_{\Omega\backslash R_{\epsilon}} \Delta w \phi dxdy - \iint_{R_{\epsilon}} \nabla w \cdot \nabla \phi dxdy + \int_{\partial R_{\epsilon}} \nabla w\cdot \mathbf{n}_\epsilon  \phi ds.
\ea
\eq
%\bq\label{uwdiff}
%\ba
%-\iint_\Omega \Delta \tilde{u} \phi dxdy = & -\iint_\Omega \Delta u \phi dxdy +\iint_\Omega \Delta W \phi dxdy \\
%=  & \iint_\Omega \delta_\gamma \phi dxdy - \iint_\Omega \nabla W \cdot \nabla \phi dxdy + \int_{\gamma} (\nabla W^- - \nabla W^+)\cdot \mathbf{n_\gamma} \phi ds\\
%= & \iint_\Omega \delta_\gamma \phi dxdy - \iint_{\Omega\backslash\gamma} \nabla W \cdot \nabla \phi dxdy + \int_{\gamma} (w_y^- - w_y^+)\phi ds\\
%= & \int_\gamma \phi ds - \iint_{\Omega\backslash\gamma} \nabla w \cdot \nabla \phi dxdy - \int_\gamma \phi ds,\\
%\ea
%\eq
For each term on the right hand side of (\ref{uwdiff}), we have the following estimates. In particular,
$$
\int_{\partial R_{\epsilon}} \nabla w\cdot \mathbf{n}_\epsilon  \phi ds =\int_{-\epsilon}^{1+\epsilon} (w_y(x,\epsilon)-w_y(x,-\epsilon))\phi dx + \int_{-\epsilon}^{\epsilon} (w_x(1+\epsilon,y)-w_x(-\epsilon,y))\phi dy.
$$
By (\ref{eq:2d}) we have
\bq
\ba
\iint_{\Omega\backslash R_{\epsilon}} \Delta w \phi dxdy=0. \nonumber% \rightarrow \iint_{\Omega\backslash \gamma} \Delta w \phi dxdy =0.
\ea
\eq
As $\epsilon \rightarrow 0$, due to the boundedness of $|\nabla w|$ in $R_{\epsilon}$, it follows
$$
\iint_{R_{\epsilon}} \nabla w \cdot \nabla \phi dxdy \rightarrow 0;
$$
and by the transmission condition in (\ref{eq:2d}), we futher have
$$
\int_{\partial R_{\epsilon}} \nabla w\cdot \mathbf{n}_\epsilon  \phi ds \rightarrow \int_0^1 (w_y(x,0+)-w_y(x,0-))\phi dx = -\int_\gamma \phi dx.
$$
Incorporating the above estimates into  equation (\ref{uwdiff}), we have
$$
-\iint_\Omega \Delta \tilde{u} \phi dxdy = 0,  \quad \forall\ \phi \in C_0^\infty(\Omega).
$$
We then conclude that
$$
-\Delta \tilde{u} = 0 \quad \text{in } \Omega.
$$
Note that $\tilde{u}=u-w=0$ on $\partial \Omega$, then it follows $\tilde{u}=0$ in $\Omega$, namely, $u=w$ in $\Omega$.

Therefore, the regularity estimates for $u$ in $\Omega^+$, $\Omega^-$, and in $B_{d, i}$, $i=1,2$ can be derived from the corresponding estimates for $w$ in Lemma \ref{waway} and in Theorem \ref{lemma3-4}. The regularity estimates for $u$ in $R_3$ follow from the results in \cite{Kondratiev67, Grisvard85} for elliptic Dirichlet problems in polygonal domains.
\end{proof}

\section{Optimal finite element methods}\label{sec-4}

According to Lemma \ref{thm2-2}, the solution of equation \eqref{eq:Possion} is merely in $H^{\frac{3}{2}-\epsilon}(\Omega)$ for any $\epsilon>0$. The singularities in the solution can severely slow down the  convergence of the usual finite element method associated with a quasi-uniform mesh. In this section, we  propose new finite element algorithms to approximate the solution of equation \eqref{eq:Possion} that shall converge at the optimal rate.

\subsection{The finite element method}

Let $\mathcal{T}=\{T_i\}$ be a triangulation of $\Omega$ with triangles. For $m\geq 1$,  we denote the Lagrange finite element space by
\bq\label{femspace}
S(\mathcal{T},m)=\{v\in C^0(\Omega) \cap H_0^1(\Omega):v|_T\in P_m(T), \ \forall \ T \in \mathcal{T}\},
\eq
where $P_m(T)$ is the space of  polynomials with degree no more than $m$ on  $T$. Following the variational form (\ref{eqn.weak}), we define  the finite element solution  $u_h\in S(\mathcal{T},m)$ of  equation \eqref{eq:Possion} by
\begin{equation}
\label{eq:FEM}
\int_\Omega\nabla u_h\cdot\nabla v_hdx=\int_\gamma v_hdx, \quad \forall\ v_h\in S(\mathcal{T},m).
\end{equation}
%Let $u_h=u_i\sum_{T\in\mathcal{T}}\varphi_i$, it follows
%$$
%\sum_{T\in\mathcal{T}} u_i  \int_\Omega \nabla \varphi_i\nabla \varphi_jdx=\int_\gamma \varphi_jdx.
%$$
Suppose that the mesh $\mathcal T$ consists of quasi-uniform triangles with size $h$. Because of the lack of regularity in the solution ($u\in H^{\frac{3}{2}-\epsilon}(\Omega))$, the standard error estimate \cite{Ciarlet74} yields only a sup-optimal convergence rate
\begin{equation}
\|u-u_h\|_{H^1(\Omega)}\leq C h^{\frac{1}{2}-\epsilon},\qquad {\rm{for}}\ \epsilon>0.
\end{equation}
This is highly ineffective since the optimal convergence rate using the $m$th-degree polynomials when the solution is smooth  is
\begin{equation*}
\|u-u_h\|_{H^1(\Omega)}\leq C h^{m}.
\end{equation*}

We now propose new finite element methods to solve equation \eqref{eq:Possion} based on the special refinement of the triangles. Recall that the singular set $\mathcal V$ includes the endpoints of $\gamma$ and all the vertices of $\Omega$. We call the points in $\mathcal V$ the \textit{singular points}.

\begin{figure}
\includegraphics[scale=0.34]{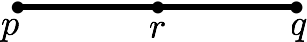}\hspace{3cm}\includegraphics[scale=0.34]{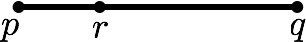}
\caption{The new node  of the an edge $pq$ (left -- right): no singular vertices (midpoint); $p$ is a singular point ($|pr|=\kappa_p|pq|$,  $\kappa_p<0.5$).}\label{fig.2} \end{figure}

\begin{figure}
 \includegraphics[scale=0.34]{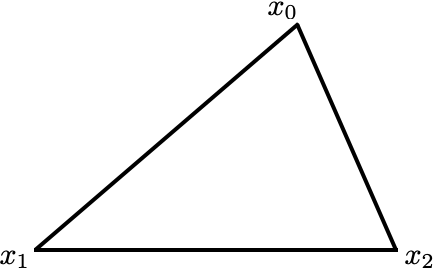}\hspace{0.5cm} \includegraphics[scale=0.34]{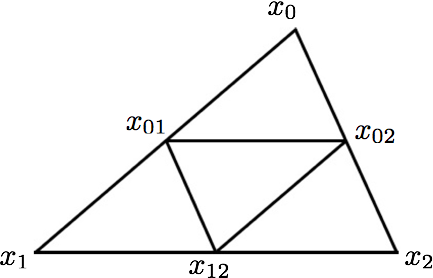}\\\hspace{0cm}\includegraphics[scale=0.34]{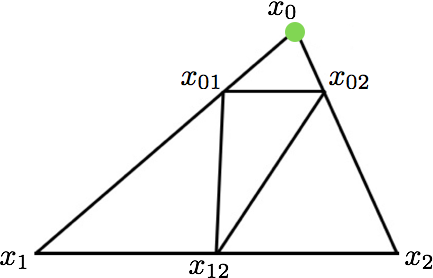}\hspace{.5cm}\includegraphics[scale=0.34]{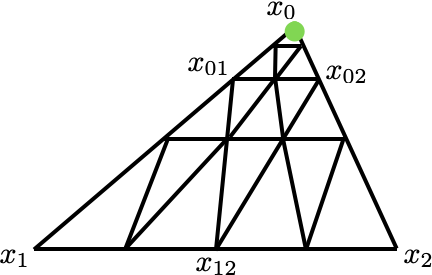}
\caption{Refinement of a triangle $\triangle x_0x_1x_2$. First row: (left -- right): the initial triangle and the midpoint refinement; second row: two consecutive graded refinements toward $x_0$, ($\kappa_{x_0}<0.5$).}\label{fig.333}\end{figure}

\begin{alg} \label{graded} (Graded refinements) Suppose each singular point  is a vertex in the triangulation $\mathcal T$ and each triangle in $\mathcal T$ contains at most one singular point. We also suppose $\mathcal T$ conforms to $\gamma$. Namely,  $\gamma$ is the union of some edges in $\mathcal T$ and  does not cross triangles in $\mathcal T$.
Let ${pq}$ be an edge in the triangulation $\mathcal T$ with $p$ and $q$ as the endpoints. Then, in a graded refinement, a new node $r$ on $pq$ is produced according to the following conditions:
\begin{itemize}
\item[1.] (Neither $p$ or $q$ is a singular point.) We choose $r$  as the midpoint ($|pr|=|qr|$).
\item[2.] ($p$ is a singular point.) We choose $r$  such that $|pr|=\kappa_p|pq|$, where $\kappa_p\in (0, 0.5)$ is a parameter that will be specified later. See Figure \ref{fig.2} for example.
\end{itemize}
Then, the graded refinement, denoted by $\kappa(\mathcal T)$, proceeds as follows.
For each triangle in $\mathcal T$, a new node is generated on each edge as described above. Then, $T$ is decomposed into four small triangles by connecting these new nodes (Figure \ref{fig.333}). Given an initial mesh $\mathcal T_0$ satisfying the condition above, the associated family of graded meshes $\{\mathcal T_n,\ n\geq0\}$ is defined recursively $\mathcal T_{n+1}=\kappa(\mathcal T_{n})$.
\end{alg}

\begin{rem}\label{rkgraded}
In Algorithm \ref{graded}, we choose the parameter $\kappa_p$ for each $p\in \mathcal V$ as follows. Recall $m$ is the degree of polynomials in the finite element space $S(\mathcal T_n, m)$. Then, if $p$ is an endpoint of $\gamma$, we choose $\kappa_p=2^{-\frac{m}{a}}$ for any $0<a<1$, and if $p$ is an vertex of the domain $\Omega$, we choose $\kappa_p<2^{-\frac{m\omega}{\pi}}$, where $\omega$ is the largest interior angle of the domain.
 \end{rem}
%\begin{definition}\label{graded}
%Let $\mathcal{T}$ be a triangulation of $\Omega$ such that no two vertices of $\Omega$ belong to the same triangle of $\mathcal{T}$. The $\kappa$ refinement of $T$, denoted by $\kappa(T)$ is obtained by dividing each side $AB$ of $T$ in two parts as follows. If neither $A$ nor $B$ is a singular vertex, then we divide $AB$ into two equal parts. Otherwise, if $A$ is a vertex, we divide $AB$ into $AC$ and $CB$ such that $|AC|=\kappa_Q|AB|$. If $\kappa_Q<0.5$, the generated mesh is called a graded triangular mesh. If $\kappa_Q=0.5$, the generated mesh is the general quasi-uniform mesh.
%\end{definition}
Let $S_n:=S(\mathcal{T}_n,m)$ be the finite element space of degree $m$ associated with the graded meshes defined in Algorithm \ref{graded} and Remark \ref{rkgraded}. Then, we define the finite element solution $u_n\in S_n$ as
\begin{equation}
\label{eq:FEM1}
a(u_n, v_n)=\int_\Omega\nabla u_n\cdot\nabla v_ndx=\int_\gamma v_ndx, \quad \forall\ v_n\in S_n.
\end{equation}
Note that the bilinear form $a(\cdot, \cdot)$ is coercive and continuous on $S_n$. Thus, by C\'ea's Theorem, we have
\begin{eqnarray}\label{cea}
\|u-u_n\|_{H^1(\Omega)}\leq C\inf_{v\in S_n}\|u-v\|_{H^1(\Omega)}.
\end{eqnarray}
In the rest of this section, we shall show that the proposed numerical solution $u_n$ converges to the solution $u$ of \eqref{eq:Possion} in the optimal rate.

\subsection{Interpolation error estimates}
Recall the three regions $R_1$, $R_2$ and $R_3$ of the domain $\Omega$ in Remark \ref{rk31}. $R_1$ is the region that is away from the singular set $\mathcal V$. $R_2$ is the region close to the endpoints of $\gamma$ and $R_3$ is the region close to the boundary of the domain. According to the regularity analysis in Section \ref{sec-3}, the solution of equation (\ref{eq:Possion}) behaves differently in these three regions. We therefore focus on the local interpolation error analysis in different regions.

%We follow the refinement of the triangular meshes in \cite{Li10,Li13} as follows.
%\begin{definition}\label{graded}
%Let $\mathcal{T}$ be a triangulation of $\Omega$ such that no two vertices of $\Omega$ belong to the same triangle of $\mathcal{T}$. The $\kappa$ refinement of $T$, denoted by $\kappa(T)$ is obtained by dividing each side $AB$ of $T$ in two parts as follows. If neither $A$ nor $B$ is a singular vertex, then we divide $AB$ into two equal parts. Otherwise, if $A$ is a vertex, we divide $AB$ into $AC$ and $CB$ such that $|AC|=\kappa_Q|AB|$. If $\kappa_Q<0.5$, the generated mesh is called a graded triangular mesh. If $\kappa_Q=0.5$, the generated mesh is the general quasi-uniform mesh.
%\end{definition}

%\begin{figure}
%\begin{center}
%\begin{tikzpicture}[scale=0.9]
%\draw[ultra thick] (-5,0) node[anchor=north east] {$Q$} -- (-2,3) -- (-1,0) -- (-5,0);
%\draw[ultra thick] (2,0) node[anchor=north east] {$Q$} -- (5,3) -- (6,0) -- (2,0);
%\draw[ultra thick] (3,0) -- (5.5,1.5) node[anchor = west] {$l_2$};%
%\draw[ultra thick] (5.5,1.5) -- (2.75,0.7) -- (3,0);
%\draw[->,ultra thick] (-0.5,1.5) -- (1.8,1.5);
%\draw (2.65,0.25) node {$l_1$};
% \draw (2.65,0.25) node {$l_2$};
%\end{tikzpicture}
%\end{center}
%\caption{One refinement of the triangle $T$ with vertex Q, $\kappa=l_1/l_2$}
%\label{figrefine}
%\end{figure}

\subsubsection{Interpolation error estimates in $R_1$ and $R_3$.}
\begin{lem}\label{r1r3}
Recall the triangulation $\mathcal T_n$ in Algorithm \ref{graded} and Remark \ref{rkgraded}. Let $T_{(0)}\in\mathcal T_{0}$ be an initial triangle and let $u_I$ be the nodal interpolation of $u$ associated with $\mathcal T_n$. If $\bar T_{(0)}$ does not contain the endpoint of $\gamma$, then
\begin{equation*}
\|u-u_I\|_{H^1(T_{(0)})}\leq Ch^m,
\end{equation*}
where $h:=2^{-n}$.
\end{lem}
\begin{proof}
Note that if $\bar T_0$ does not contain the endpoint of $\gamma$, then $\bar T_{(0)}\cap\mathcal V= \emptyset$ or $\bar T_{(0)}$ contains a vertex of the domain $\Omega$. If $\bar T_{(0)}\cap\mathcal V=\emptyset$, we have $u\in H^{m+1}(T_{(0)})$ (Theorem \ref{ureg}) and the mesh on $T_{(0)}$ is quasi-uniform (Algorithm \ref{graded}) with size $O(2^{-n})$. Therefore, based on the standard interpolation error estimate, we have
\begin{eqnarray}\label{1.1}
\|u-u_I\|_{H^1(T_{(0)})}\leq Ch^m\|u\|_{H^{m+1}(T_{(0)})}\leq Ch^m.
\end{eqnarray}
In the case that $\bar T_0$ contains a vertex of the domain, the solution may be singular in the neighborhood of a corner. Based on the results in \cite{BNZ1}, the solution $u\in \mathcal K^{m+1}_{a+1}(T_{(0)})$ for $a<\frac{\pi}{\omega}$ and $m\geq 1$, where $\omega$ is the largest interior angle of the domain. Note that the graded mesh on $T_{(0)}$ with the  parameter in Remark \ref{rkgraded} is the same mesh defined in \cite{BNZ1, Li10}, which can recover the optimal convergence rate in the finite element method even when the solution has corner singularities:
\begin{eqnarray}\label{1.2}
\|u-u_I\|_{H^1(T_{(0)})}\leq  Ch^m.
\end{eqnarray}
The proof is hence completed by (\ref{1.1}) and (\ref{1.2}).
\end{proof}

\subsubsection{Interpolation error estimates in $R_2$.} We now study the interpolation error in the neighborhood of the endpoint $Q$ of $\gamma$. In the rest of this subsection, we assume $T_{(0)}\in\mathcal T_0$ is an initial triangle such that $Q$ is a vertex of $T$. According to Remark \ref{rkgraded}, the mesh on $T_{(0)}$ is graded toward $Q$ with $\kappa_Q=2^{-\frac{m}{a}}$ for any $0<a<1$. We first define mesh layers on $T_0$ which are collections of triangles in $\mathcal T_n$.
\begin{definition} (Mesh layers) Let $T_{(i)}\subset T_{(0)}$ be the triangle in $\mathcal T_i$, $0\leq i\leq n$, that is attached to the singular vertex $Q$ of $T_{(0)}$. For $0\leq i<n$, we define the $i$th mesh layer of $\mathcal T_n$ on $T_{(0)}$ to be the region $L_{i}:=T_{(i)}\setminus T_{(i+1)}$; and for $i=n$, the $n$th layer is $L_{n}:=T_{(n)}$.  See Figure \ref{fig.layer} for example.
\end{definition}
\begin{figure}
 \includegraphics[scale=0.34]{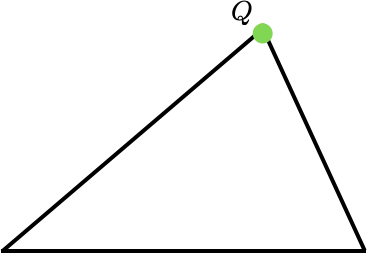}\hspace{0.5cm} \includegraphics[scale=0.34]{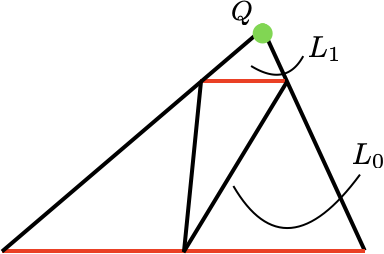}\hspace{0.5cm}\includegraphics[scale=0.34]{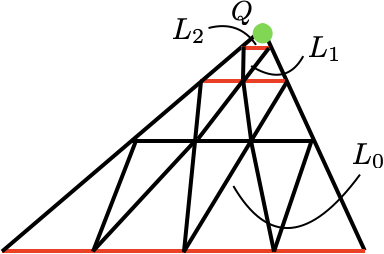}
\caption{Mesh layers (left -- right): the initial triangle $T_{(0)}$ with a vertex $Q$; two layers after one refinement; three layers after two refinements.}\label{fig.layer}\end{figure}

\begin{rem}
The triangles in $\mathcal T_n$ constitute $n$ mesh layers on $T_{(0)}$. According to Algorithm \ref{graded} and the choice of grading parameters in Remark \ref{rkgraded}, the mesh size in the $i$th  layer $L_i$ is
\begin{equation}\label{eqn.size}O(\kappa_Q^i2^{i-n}). \end{equation}
Meanwhile, the weight function $\rho$ in (\ref{eqn.rho}) satisfies
\begin{eqnarray}\label{eqn.dist}
\rho=O(\kappa_Q^i) \ \ \ {\rm{in\ }} L_i\ (0\leq i< n) \qquad {\rm{and}}  \qquad \rho \leq C\kappa_Q^n \ \ \ {\rm{in\ }} L_n.
\end{eqnarray}
Although the mesh size varies in different layers, the triangles in $\mathcal T_n$ are shape regular. In addition, using the local Cartesian coordinates such that $Q$ is the origin,  the mapping
\begin{eqnarray}\label{eqn.map}
\mathbf B_{i}= \begin{pmatrix}
  \kappa_Q^{-i}   &   0 \\
  0    &   \kappa_Q^{-i} \\
\end{pmatrix},\qquad 0\leq i\leq n
\end{eqnarray}
is a bijection between $L_i$ and $L_0$ for $0\leq i<n$ and a bijection between $L_n$ and $T_{(0)}$.
\end{rem}

We then derive the interpolation error estimate in each layer.

\begin{lem}\label{TNtri}
Recall $\kappa_Q = 2^{-\frac{m}{a}}$ for the graded mesh on $T_{(0)}$, $m\geq 1$ and $0<a<1$. Let $u_I$ be the nodal interpolation of $u$ in the $i$th layer $L_i$ on $T_{(0)}$, $0\leq i<n$. Then, for $h:=2^{-n}$, we have
$$
|u-u_{I}|_{H^1(L_i)} \leq Ch^m.
$$
\end{lem}
\begin{proof} Based on Theorem \ref{ureg}, the solution can be decomposed into two parts on $T_{(0)}$, $u=u_{reg}+u_s$, where for $m\geq 1$ and $0<a<1$, \begin{equation*}
\|u_{reg}\|_{\maK_{a+1}^{m+1}(T_{(0)})}+\|u_s\|_{L^\infty(T_{(0)})} \leq C.
\end{equation*}
Since $u_s$ belongs to a finite dimensional space, the norms of $u_s$ are equivalent. Thus, we have
\begin{equation}\label{eqn.decomp}
\|u_{reg}\|_{\maK_{a+1}^{m+1}(T_{(0)})}+\|u_s\|_{H^{m+1}(T_{(0)})} \leq C.
\end{equation}

Note that in each $L_i$, $i<n$, the space $\maK_{a+1}^{m+1}$ is equivalent to $H^{m+1}$. Therefore, both $u_{reg}$ and $u_s$ are continuous functions in $L_i$. Let $u_{reg, I}$ and $u_{s,I}$ be the nodal interpolations of $u_{reg}$ and $u_{s}$, respectively. Then, it is clear that $u_I=u_{reg, I}+u_{s,I}$. Thus, we have
\begin{eqnarray}\label{eqn.deom1}
|u-u_{I}|_{H^1(L_i)}\leq |u_{reg}-u_{reg, I}|_{H^1(L_i)}+|u_{s}-u_{s, I}|_{H^1(L_i)}.
\end{eqnarray}
We shall obtain the estimate for each term on the right hand side of (\ref{eqn.deom1}).

Recall the mapping $\mathbf B_i$ in (\ref{eqn.map}). For any point $(x, y)\in L_i$, let $(\hat x, \hat y)=\mathbf B_i(x,y)\in L_0$. Then, for a function $v(x, y)$  in $L_i$, define $\hat v(\hat x, \hat y):=v(x, y)$ in $L_0$. Using the standard interpolation error estimate, the scaling argument,  the estimate in (\ref{eqn.size}) and the mapping in (\ref{eqn.map}), we have
\begin{eqnarray*}
 |u_{reg}-u_{reg, I}|_{H^1(L_i)}&=& |\hat u_{reg}-\hat u_{reg, \hat I}|_{H^1(L_0)}\leq C2^{(i-n)m}|\hat u_{reg}|_{H^{m+1}(L_0)}\\
 &\leq& C 2^{(i-n)m}\kappa_Q^{mi}|u_{reg}|_{H^{m+1}(L_i)}=C h^m(2\kappa_Q)^{mi}|u_{reg}|_{H^{m+1}(L_i)}.
\end{eqnarray*}
Recall $\kappa_Q<2^{-\frac{m}{a}}$ for any $0<a<1$ and recall the estimate in (\ref{eqn.dist}). Then, continuing the estimate above, we obtain
\begin{eqnarray}
 |u_{reg}-u_{reg, I}|^2_{H^1(L_i)}&\leq&C h^{2m}\sum_{|\alpha|=m+1}|\rho^{-1-a}\rho^{m+1}\partial^\alpha u_{reg}|^2_{L^2(L_i)}\nonumber\\&\leq& Ch^{2m} \|u_{reg}\|_{\maK_{a+1}^{m+1}(L_i)}^2,\label{eqn.d1}
\end{eqnarray}
where the last step is based on definition of the weighted space.

For $|u_{s}-u_{s, I}|_{H^1(L_i)}$, by the fact that  $\kappa_Q<0.5$, we similarly have
\begin{eqnarray}
 |u_{s}-u_{s, I}|_{H^1(L_i)}&=& |\hat u_{s}-\hat u_{s, \hat I}|_{H^1(L_0)}\leq C2^{(i-n)m}|\hat u_{s}|_{H^{m+1}(L_0)}\nonumber\\
 &\leq& C 2^{(i-n)m}\kappa_Q^{mi}|u_{s}|_{H^{m+1}(L_i)}=C h^m|u_{s}|_{H^{m+1}(L_i)}.\label{eqn.d2}
\end{eqnarray}
Then, the proof is completed by combining (\ref{eqn.deom1}), (\ref{eqn.d1}), (\ref{eqn.d2}), and (\ref{eqn.decomp}).
\end{proof}

We now derive the interpolation error estimate in the last layer $L_n$ on $T_{(0)}$.

\begin{lem}\label{TNtri2}
Recall $\kappa_Q = 2^{-\frac{m}{a}}$ for the graded mesh on $T_{(0)}$, $m\geq 1$ and $0<a<1$. Let $u_I$ be the nodal interpolation of $u$ in the $n$th layer $L_n$ on $T_{(0)}$ for $n$ sufficiently large. Then, for $h:=2^{-n}$, we have
$$
|u-u_{I}|_{H^1(L_n)} \leq Ch^m.
$$
\end{lem}
\begin{proof} Recall from Theorem \ref{ureg} that on $T_{(0)}$, $u=u_{reg}+u_s\in \mathcal K_{a+1}^{m+1}+W$ (see also (\ref{eqn.decomp})). Let $u_{reg, I}$ and $u_{s,I}$ be the nodal interpolations of $u_{reg}$ and $u_{s}$, respectively. Recall $u_s$ is a constant in the $n$th layer $L_n$ when $n$ is sufficiently large, and therefore $(u_s-u_{s, I})|_{L_n}=0$. Thus, it is sufficient to estimate $|u_{reg}-u_{reg, I}|_{H^1(L_n)}$.

Recall the mapping $\mathbf B_n$ in (\ref{eqn.map}). For any point $(x, y)\in L_n$, let $(\hat x, \hat y)=\mathbf B_n(x,y)\in T_{(0)}$. Then, for a function $v(x, y)$ in $L_n$, define $\hat v(\hat x, \hat y):=v(x, y)$ in $T_{(0)}$.
Let $\psi: T_{(0)} \rightarrow [0, 1]$ be a  smooth function that
is equal to $0$ in a neighborhood of $Q$, but is equal
to 1 at all the other nodal points in $\mathcal T_0$.  Then, we let  $w=\psi \hat u_{reg}$ in $T_{(0)}$. Consequently, we have for $l\geq 0$
\begin{equation}\label{eqn.aux111}
 	\|w\|^2_{\maK^{l}_{1}(T_{(0)})}=\|\psi
	\hat u_{reg}\|^2_{\maK^{l}_{1}(T_{(0)})} \leq C
	\|\hat u_{reg}\|^2_{\maK^{l}_{1}(T_{(0)})},
\end{equation}
where $C$ depends on $k$ and the smooth function $\psi$. Moreover, the condition $u_{reg}\in \maK_{a+1}^{m+1}(T_{(0)})$ implies $u_{reg}(Q)=0$. Let $w_{\hat I}$ be the nodal interpolation of $w$ associated with the mesh $\mathcal T_0$ on $T_{(0)}$.
Therefore, by the definition of $w$, we have
\begin{eqnarray}\label{wi} w_{\hat I}=\hat u_{reg, \hat I} = \widehat{u_{reg, I}} \quad {\rm{in}}\ T_{(0)}.\end{eqnarray}

Note that the $\maK^{l}_{1}$ norm and the $H^m$ norm are equivalent for $w$ on $T_{(0)}$, since $w=0$ in the neighborhood of the vertex $Q$. Let $r$ be the distance to $Q$. Then, by the definition of the weighted space, the scaling argument, (\ref{eqn.aux111}),  (\ref{wi}),  and (\ref{eqn.dist}), we have
\begin{eqnarray*}
|u_{reg}-u_{reg, I}|_{H^1(L_{n})}^2 &\leq& C\|u_{reg}-u_{reg, I}\|_{\maK^1_{{1}}(L_{n})}^2 \leq C\sum_{|\alpha|\leq 1}\|r(x,y)^{|\alpha|-1}\partial^\alpha (u_{reg}- u_{reg,I})\|_{L^2(L_{n})}^2\\
&=&C\sum_{|\alpha|\leq 1}\|r(\hat{x},\hat{y})^{|\alpha|-1}\partial^\alpha (\hat u_{reg}- \widehat{u_{reg,I}})\|_{L^2(T_{(0)})}^2\leq C\|\hat u_{reg}- w+w-\widehat{u_{reg,I}}\|_{\maK^1_{1}( T_{(0)})}^2 \\
& \leq &C\big( \|\hat u_{reg}-w\|^2_{\maK^1_{1}(T_{(0)})} +
\|w-\widehat{u_{reg,I}}\|^2_{\maK^1_{1}( T_{(0)}  )}\big)  \\
& = & C\big( \|\hat u_{reg}-w\|^2_{\maK^1_{1}(T_{(0)})} +
\|w-w_{\hat I}\|^2_{\maK^1_{1}( T_{(0)} )}\big)  \\
& \leq & C\big(   \|\hat u_{reg}\|^2_{\maK^1_{1}(T_{(0)} )} +
\|w\|^2_{\maK^{m+1}_{1}( T_{(0)}  )}\big) \leq C\big(   \|\hat u_{reg}\|^2_{\maK^1_{1}(T_{(0)})} +
\|\hat u_{reg}\|^2_{\maK^{m+1}_{1}( T_{(0)} )}\big)\\%
& \leq & C\big(   \|u_{reg}\|^2_{\maK^1_{1}(L_n)} +
\|u_{reg}\|^2_{\maK^{m+1}_{1}( L_n)}\big)\leq C
\kappa_Q^{2na}\|u_{reg}\|_{\maK^{m+1}_{{a}+1}(L_{n})}^2\\
& \leq & C
2^{-2nm}\|u_{reg}\|_{\maK^{m+1}_{{a}+1}(L_{n})}^2\leq C
h^{2m}.
\end{eqnarray*}
This completes the proof.
\end{proof}

Therefore, for  the  finite element method solving equation (\ref{eq:Possion}) defined in Algorithm \ref{graded} and Remark \ref{rkgraded}, we obtain the optimal convergence rate.
\begin{thm}\label{thm.optimal} Let $S_n$ be the finite element space associated with the graded triangulation $\mathcal T_n$ defined in Algorithm \ref{graded} and Remark \ref{rkgraded}. Let $u_n\in S_n$ be the finite element solution of equation (\ref{eq:Possion}) defined in (\ref{eq:FEM1}). Then,
$$
\|u-u_n\|_{H^1(\Omega)}\leq C{\rm{dim}}(S_n)^{-\frac{m}{2}},
$$
where dim$(S_n)$ is the dimension of $S_n$.
\end{thm}
\begin{proof}
By C\'ea's Theorem (see (\ref{cea})),
$$
\|u-u_n\|^2_{H^1(\Omega)}\leq C\|u-u_I\|^2_{H^1(\Omega)}=C\sum_{T_{(0)}\in\mathcal T_0}\|u-u_I\|^2_{H^1(T_{(0)})}.
$$
Based on the Poincar\'e inequality and Lemmas \ref{TNtri} and \ref{TNtri2}, if the initial triangle $T_{(0)}$ has an endpoint of $\gamma$ as a vertex, we have
$$
\|u-u_I\|^2_{H^1(T_{(0)})}\leq Ch^m=C2^{-mn}.
$$
Summing up this estimate and the estimates in Lemma \ref{r1r3}, and noting  that based on Algorithm \ref{graded} dim$S_n=O(4^n)$, we obtain
$$
\|u-u_n\|^2_{H^1(\Omega)}\leq Ch^m\leq C{\rm{dim}}(S_n)^{-\frac{m}{2}},
$$
which completes the proof.
\end{proof}

\begin{rem} The solution of equation (\ref{eq:Possion}) may possess singularities across the line segment $\gamma$, near the vertices of the domain, and near the endpoints of $\gamma$. We have derived  regularity results in weighted Sobolev spaces and  proposed numerical methods that solve equation (\ref{eq:Possion}) in the optimal convergence rate. These results can be extended to more general cases, for example, the case where the line fracture is replaced by multiple line fractures, whether intersecting or non-intersecting. With proper modifications, we also expect the analytical tools will be useful when $\gamma$ is a smooth curve and when the source term $\delta_{\gamma}$ is replace by $q\delta_{\gamma}$ for  $q\in L^2({\gamma})$.
\end{rem}

\section{Numerical examples}
\label{sec-5}
In this section, we present numerical test results to validate our theoretical predictions for the proposed finite element method solving equation (\ref{eq:Possion}). Since the solution $u$ is unknown, we use the following numerical convergence rate
%, the function $q$ and the fracture $\gamma$ will be specified later.
\begin{eqnarray}\label{rate}
e=\log_2\frac{|u_j-u_{j-1}|_{H^1(\Omega)}}{|u_{j+1}-u_j|_{H^1(\Omega)}},
\end{eqnarray}
where $u_j$ is the finite element solution on the  mesh $\mathcal T_j$ obtained after $j$  refinements of the initial triangulation $\mathcal T_0$. According to Theorem \ref{thm.optimal}, when the  optimal convergence rate is obtained, the value of $e$ shall be close to $m$, where $m$ is the degree of the polynomial used in the numerical method. This desired rate can be achieved especially when the grading parameter near the endpoint $Q$ of $\gamma$ satisfies $\kappa_Q=2^{-\frac{m}{a}}$ for any $0<a<1$ and the grading parameter near a vertex $p$ of domain satisfies $\kappa_p<2^{-\frac{m\omega}{\pi}}$, where $\omega$ is the largest  interior angle among all the vertices of $\Omega$.

For Example \ref{P1h} and \ref{P1ex2}, we consider the finite element method based on $P_1$ polynomials for problem (\ref{eq:Possion}) in a square domain $\Omega=(0,1)^2$.

\begin{example}\label{P1h} (Union-Jack meshes and graded meshes)
In this example,  the line fracture $\gamma=Q_1Q_2$ has two   vertices $Q_1=(0.25,0.5)$ and $Q_2=(0.75,0.5)$. We use  finite element methods on
%and $\gamma=[0.05,0.95]\times \{0.5\}$ and
two types of triangular meshes:  the Union-Jack mesh with  elements across the line fracture $\gamma$; and the graded meshes conforming to $\gamma$ defined in Algorithm \ref{graded} with different values of the grading parameter. The initial triangulations are given in (a) and (c) of Figure \ref{Mesh_Init}, respectively, where the Union-Jack mesh has $128$ elements and the graded mesh has $64$ elements. To refine the Union-Jack mesh, each triangle is divided into four equal triangles.

Note that in the square domain, the vertices of the domain does not lead to corner singularities in $H^2$. Therefore, we  use  quasi-uniform meshes near the corners, which shall not affect the global convergence rate. However,  in the region across $\gamma$, the solution  merely belongs to $H^{\frac{3}{2}-\epsilon}$ for any $\epsilon>0$. Union-Jack mesh does not resolve the singularity across the fracture $\gamma$. Thus,  on the Union-Jack mesh, the convergence rate (\ref{rate}) of the numerical solution shall be about $0.5$. The graded mesh conforms to $\gamma$ and therefore resolves the solution singularity across $\gamma$. Based on Theorem \ref{thm.optimal}, when the grading parameter  for the endpoints of $\gamma$ satisfies $\kappa:=\kappa_{Q_1}=\kappa_{Q_2}=2^{-\frac{1}{a}}<0.5$, the singular solution near $Q_1$ and $Q_2$ shall be well approximated, which yields the optimal convergence rate in the numerical approximation.

\begin{figure}
\centering
\subfigure[]{\includegraphics[width=0.222\textwidth]{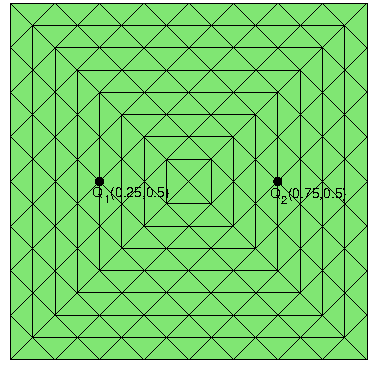}}\hspace{0.19cm}
\subfigure[]{\includegraphics[width=0.22\textwidth]{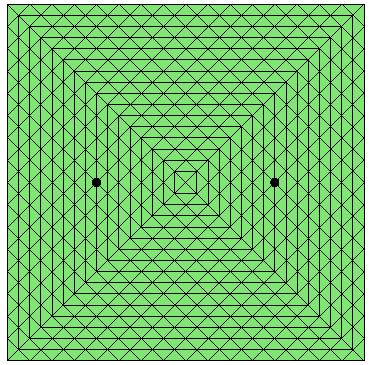}}\hspace{0.3cm}
\subfigure[]{\includegraphics[width=0.22\textwidth]{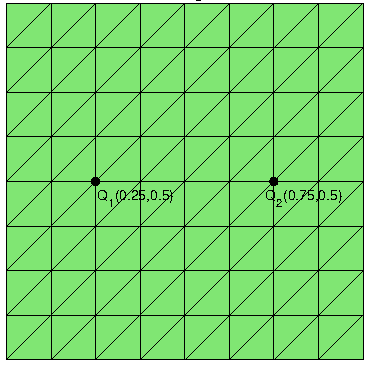}}\hspace{0.3cm}
\subfigure[]{\includegraphics[width=0.22\textwidth]{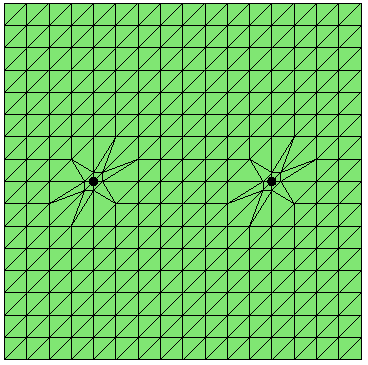}}
\caption{Graded mesh and Union-Jack mesh. (a) and (b): the initial Union-Jack mesh and the mesh after one refinement. (c) and (d): the initial graded mesh and the mesh after one refinement, $\kappa=\kappa_{Q_1}=\kappa_{Q_2}=0.2$. }\label{Mesh_Init}
\end{figure}

The convergence rate (\ref{rate}) associated with these two types of meshes are reported in Table \ref{TabConRate}. The first five rows are the rates on graded meshes, and the last row contains data on the Union-Jack mesh. Here $j$ is the number of refinements from the initial mesh. It is clear that the  rate  on a sequence of Union-Jack meshes is suboptimal with  $e=0.5$.  For graded meshes, when  $\kappa<0.5$, the convergence rate is  optimal with rate $e=1$; and the convergence is not optimal when  $\kappa=0.5$. These results are closely aligned with our aforementioned theoretical predication.

\begin{table}[!htbp]\tabcolsep0.03in
\centering
\caption{Convergence history of the numerical solution in Example \ref{P1h} with mesh refinements.}
\begin{tabular}{|l|l|} %{lclclclclcl}
\hline
  $\kappa \backslash j$             & $j=2$               $j=3$       $j=4$               $j=5$         \\
\hline
$\kappa=0.1$   & 0.99  \hspace{0.07cm}     0.94     \hspace{0.07cm}        0.97     \hspace{0.07cm}     0.99 \\
\hline
$\kappa=0.2$   & 0.97   \hspace{0.07cm}   0.99    \hspace{0.07cm}        0.99       \hspace{0.07cm}   1.00 \\
\hline
%$\kappa=0.25$  & 0.92     & 0.98          &  1.00        & 1.00 \\
%\hline
$\kappa=0.3$   & 0.87   \hspace{0.07cm}    0.96      \hspace{0.07cm}       0.99    \hspace{0.07cm}      1.00 \\
\hline
$\kappa=0.4$   & 0.86  \hspace{0.07cm}    0.91    \hspace{0.07cm}         0.94     \hspace{0.07cm}     0.98 \\
\hline
%$\kappa=0.45$   & 0.88     & 0.91          &  0.93        & 0.94 \\
%\hline
$\kappa=0.5$   & 0.84   \hspace{0.07cm}    0.87     \hspace{0.07cm}       0.89     \hspace{0.07cm}     0.91 \\
\hline
\text{Union-Jack} & 0.46  \hspace{0.07cm}    0.47   \hspace{0.07cm}     0.49   \hspace{0.07cm}    0.49 \\
\hline
\end{tabular}\label{TabConRate}
\end{table}

\end{example}

\begin{example}\label{P1ex2} (Graded meshes for different fractures)
This example is to test the convergence rate on a sequence of graded meshes for problem (\ref{eq:Possion}) with the line fracture(s) at different locations. We shall use the linear finite element method and the same square domain as in  Example \ref{P1h} for all the numerical tests in this example.

\noindent\textbf{Test  1.}  Suppose we have a longer line fracture $\gamma=Q_1Q_2$ with two vertices $Q_1=(0.1, 0.5)$, $Q_2=(0.9,0.5)$. See Figure \ref{Mesh_Init2} for the initial mesh and the  graded mesh with $\kappa=0.2$ after four  refinements. The convergence rates associated with different values of $\kappa=\kappa_{Q_1}=\kappa_{Q_2}$ are reported in the second column of Table \ref{TabConRate2}. Similar to the numerical tests in Example \ref{P1h}, these results show that the  convergence rate is suboptimal with $e=0.93$ on the quasi-uniform mesh ($\kappa=0.5$), but becomes optimal  ($e=1$) on graded meshes for $\kappa<0.5$. %The contour of the numerical solution can be found in (c) of Figure \ref{Mesh_Init}.

\begin{figure}
\centering
\subfigure[]{\includegraphics[width=0.26\textwidth]{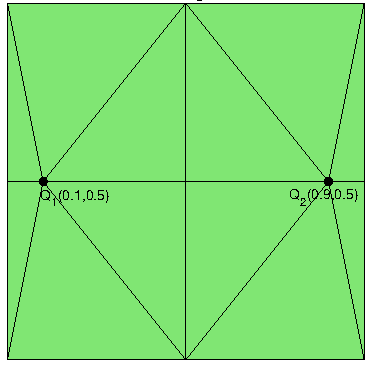}}\hspace{0.8cm}
\subfigure[]{\includegraphics[width=0.26\textwidth]{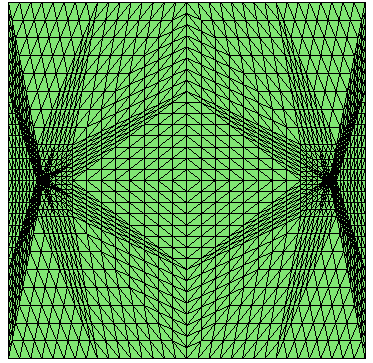}}\hspace{0.8cm}
\subfigure[]{\includegraphics[width=0.3\textwidth]{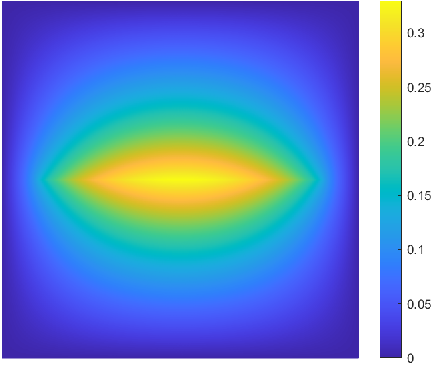}}
\caption{Graded meshes with line fracture $\gamma=Q_1Q_2$, $Q_1=(0.1,0.5)$, $Q_2=(0.9,0.5)$. (a) the initial mesh; (b) the mesh after four refinements, $\kappa=\kappa_{Q_1}=\kappa_{Q_2}=0.2$; (c) the numerical solution.}\label{Mesh_Init2}
\end{figure}

\begin{table}[!htbp]\tabcolsep0.03in
\caption{Convergence history   in Tests 1 \& 2 of Example \ref{P1ex2} on graded meshes.}
\label{TabConRate2}
\centering
\begin{tabular}{|l|l|l|}  %{lclclclclcl}
\hline
  $\kappa \backslash j$             & $j=4$              $j=5$       $j=6$             $j=7$  &      $j=4$              $j=5$       $j=6$             $j=7$  \\
\hline
$\kappa=0.1$   &  0.97 \hspace{0.07cm}   0.98  \hspace{0.07cm}  0.99 \hspace{0.07cm}  1.00 & 0.97  \hspace{0.07cm}  0.99  \hspace{0.07cm}  0.99  \hspace{0.07cm}  1.00 \\
\hline
$\kappa=0.2$   &  0.98  \hspace{0.07cm}  0.99  \hspace{0.07cm}  1.00 \hspace{0.07cm}  1.00 & 0.97  \hspace{0.07cm}  0.99  \hspace{0.07cm}  1.00  \hspace{0.07cm}  1.00  \\
\hline
$\kappa=0.3$   &  0.99 \hspace{0.07cm}   1.00  \hspace{0.07cm}  1.00 \hspace{0.07cm}  1.00 &   1.00  \hspace{0.07cm}  1.00  \hspace{0.07cm}  1.00  \hspace{0.07cm}  1.00 \\
\hline
$\kappa=0.4$   &  0.95  \hspace{0.07cm}  0.97  \hspace{0.07cm}  0.98  \hspace{0.07cm} 0.99 & 0.96  \hspace{0.07cm}  0.98  \hspace{0.07cm}  0.99  \hspace{0.07cm} 0.99\\
\hline
%$\kappa=0.45$  &  0.94  &  0.95   & 0.96  & 0.97 \\
%\hline
$\kappa=0.5$   &  0.91 \hspace{0.07cm}   0.92  \hspace{0.07cm}  0.93 \hspace{0.07cm}  0.93 & 0.93  \hspace{0.07cm}  0.93  \hspace{0.07cm}  0.94  \hspace{0.07cm}  0.94 \\
\hline
% & Test 1 & Test 2\\
 %\hline
\end{tabular}
\end{table}

%\begin{figure}
%\centering
%\subfigure{\includegraphics[width=0.26\textwidth]{figure/gmc2.png}}
%\caption{The contour of numerical solution.}
%\label{ContourFEM2}
%\end{figure}

\noindent\textbf{Test 2.} We consider a line fracture $\gamma=Q_1Q_2$ with the two vertices $Q_1=(0.2, 0.2)$, $Q_2=(0.8,0.8)$.
Here we solve the problem (\ref{eq:Possion})  on graded meshes with the initial triangulation given in Figure \ref{Mesh_Init3}. The convergence rate is reported in the third column of Table \ref{TabConRate2}. We observe that convergence rate is suboptimal with $e=0.94$ on quasi-uniform mesh ($\kappa=0.5$), but it is optimal  ($e=1$) on graded meshes for $\kappa<0.5$. The results in Table \ref{TabConRate2}, both from Test 1 and Test 2, are well  predicted by the theory as discussed above.

\begin{figure}
\centering
\subfigure[]{\includegraphics[width=0.26\textwidth]{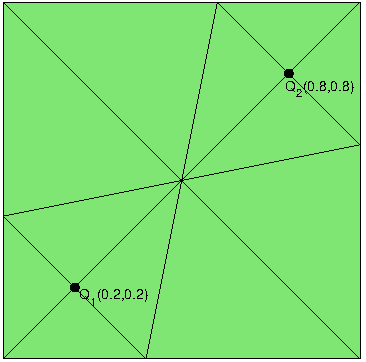}}\hspace{0.8cm}
\subfigure[]{\includegraphics[width=0.26\textwidth]{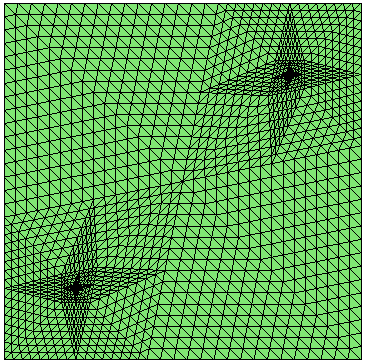}}\hspace{0.8cm}
\subfigure[]{\includegraphics[width=0.31\textwidth]{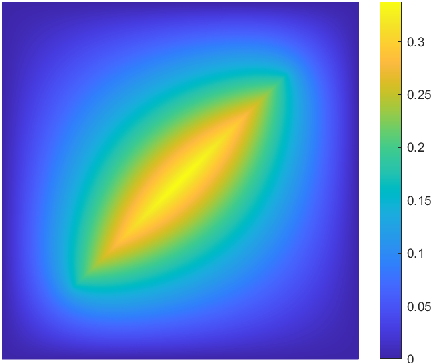}}
%\subfigure[]{\includegraphics[width=0.325\textwidth]{figure/GM_first2.eps}}
\caption{Graded meshes with line fracture $\gamma=Q_1Q_2$, $Q_1=(0.2,0.2)$, $Q_2=(0.8,0.8)$. (a) the initial mesh; (b) the mesh after four refinements, $\kappa=\kappa_{Q_1}=\kappa_{Q_2}=0.2$; (c) the numerical solution.}\label{Mesh_Init3}
\end{figure}

%\begin{table}[!htbp]\tabcolsep0.03in
%\caption{Convergence rate of the $H^1$ error of the numerical solution after $n$ mesh refinements.}
%\label{TabConRate3}
%\centering
%\begin{tabular}{lclclclclcl}
%\hline
%$\kappa \backslash n$  &  $n=4$  &  $n=5$  &  $n=6$  &  $n=7$  \\
%\hline
%$\kappa=0.1$   &  0.97  &  0.99  &  0.99  &  1.00  \\
%\hline
%$\kappa=0.2$   &  0.97  &  0.99  &  1.00  &  1.00  \\
%\hline
%$\kappa=0.3$   &  1.00  &  1.00  &  1.00  &  1.00  \\
%\hline
%$\kappa=0.4$   &  0.96  &  0.98  &  0.99  &  0.99  \\
%\hline
%$\kappa=0.45$  &  0.95  &  0.96  &  0.97  &  0.97  \\
%\hline
%$\kappa=0.5$   &  0.93  &  0.93  &  0.94  &  0.94  \\
%\hline
%\end{tabular}
%\end{table}

%\begin{figure}
%\centering
%\subfigure{\includegraphics[width=0.49\textwidth]{figure/GM_Contour3.eps}}
%\caption{The contour of numerical solution.}
%\label{ContourFEM3}
%\end{figure}

\noindent\textbf{Test 3.}  In this test, we consider two line fractures with $\gamma_1 = Q_1Q_2, \gamma_2=Q_3Q_4$ in equation (\ref{eq:Possion}). Here the vertices are $Q_1=(0.3,0.1)$, $Q_2=(0.3, 0.9)$, $Q_3=(0.6,0.1)$ and $Q_4=(0.9, 0.9)$. The initial mesh is given in Figure \ref{Mesh_Init4}.
Although two line fractures are imposed, we observe similar convergence rates: the suboptimal convergence rate with $e=0.94$ on quasi-uniform meshes ($\kappa=0.5$), and  optimal  ($e=1$) on graded meshes as $\kappa:=\kappa_{Q_1}=\kappa_{Q_2}=\kappa_{Q_3}=\kappa_{Q_4}<0.5$.

\begin{figure}
\centering
\subfigure[]{\includegraphics[width=0.26\textwidth]{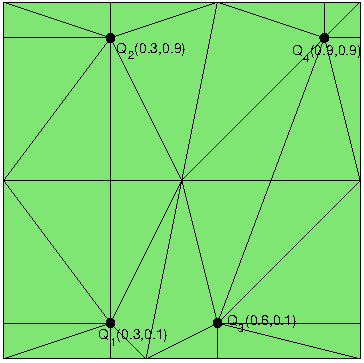}}\hspace{0.8cm}
\subfigure[]{\includegraphics[width=0.26\textwidth]{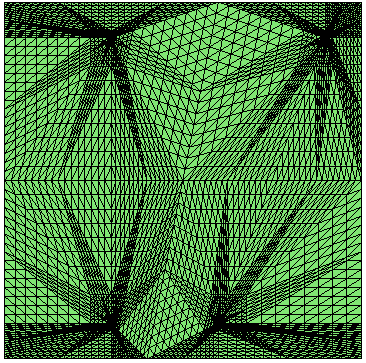}}\hspace{0.8cm}
\subfigure[]{\includegraphics[width=0.305\textwidth]{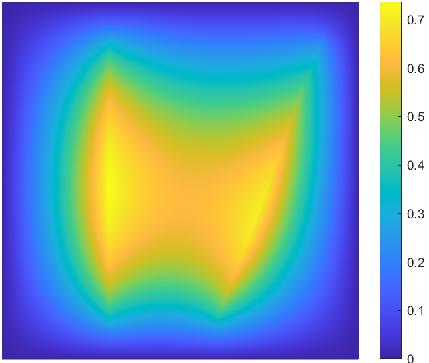}}
\caption{Graded meshes with two line fractures $\gamma_1=Q_1Q_2$ and $\gamma_2=Q_3Q_4$. (a) the initial mesh; (b) the mesh after four refinements, $\kappa=\kappa_{Q_1}=\kappa_{Q_2}=\kappa_{Q_3}=\kappa_{Q_4}=0.2$; (c) the numerical solution.}\label{Mesh_Init4}
\end{figure}

\begin{table}[!htbp]\tabcolsep0.03in
\caption{Convergence history   in Tests 3 of Example \ref{P1ex2} on graded meshes.}
\label{TabConRate4}
\centering
\begin{tabular}{|l|l|} %{lclclclclcl}
\hline
  $\kappa \backslash j$             & $j=4$               $j=5$       $j=6$               $j=7$         \\
\hline
$\kappa=0.1$   & 0.98   \hspace{0.07cm}    0.99     \hspace{0.07cm}       1.00   \hspace{0.07cm}       1.00 \\
\hline
$\kappa=0.2$   &   1.00  \hspace{0.07cm}  1.00  \hspace{0.07cm}   1.00 \hspace{0.07cm}   1.00 \\
\hline
$\kappa=0.3$   &   0.99 \hspace{0.07cm}    1.00 \hspace{0.07cm}  1.00 \hspace{0.07cm}  1.00 \\
\hline
$\kappa=0.4$   & 0.96  \hspace{0.07cm}   1.00  \hspace{0.07cm}   1.00 \hspace{0.07cm}  1.00 \\
\hline
%$\kappa=0.45$  & 0.94  &  0.96   & 0.97 & 0.97 \\
%\hline
$\kappa=0.5$   &  0.92  \hspace{0.07cm}   0.93  \hspace{0.07cm}   0.93 \hspace{0.07cm}  0.94 \\
\hline
\end{tabular}
\end{table}

In Test 1 and Test 2, we have implemented linear finite element methods proposed in Algorithm \ref{graded}. These numerical test results are in strong support of the estimate in Theorem \ref{thm.optimal}. We chose the square domain to avoid the possible corner singularity due to the non-smoothness of the domain, so that we can concentrate on the singular solution in the neighborhood of the line fracture. For general polygonal domains, the  corner singularities should be taken into account. A proper refinement algorithm near these corners are also given in Remark \ref{rkgraded} and Theorem \ref{thm.optimal}.

%\begin{figure}
%\centering
%\subfigure{\includegraphics[width=0.49\textwidth]{figure/GM_Contour4.eps}}
%\caption{The contour of numerical solution.}
%\label{ContourFEM4}
%\end{figure}

\end{example}

\begin{example}\label{ex.3} ($P_2$ finite element methods)
In this example, we consider the finite element method based on $P_2$ polynomials for equation (\ref{eq:Possion}). To minimize the effect of potential corner singularities, we solve the equation  in the triangle domain $\Omega=\Delta ABC$ with $A=(0,0), B=(1,0)$ and $C=(0.5,1)$ and the line fracture $\gamma=Q_1Q_2$ with the two vertices $Q_1=(0.3, 0.25)$, $Q_2=(0.7,0.25)$. Since all the interior angles of $\Omega$ are less then $\frac{\pi}{2}$, the solution is in $H^3$ except for the region that contains $\gamma$.  See Figure \ref{Mesh_InitP2} for the initial triangulation that conforms to the fracture.  Based on Theorem \ref{thm.optimal}, to achieve the optimal convergence rate in the numerical approximation, it is sufficient to use quasi-uniform meshes near the vertices of the domain and use graded meshes with the grading parameter $\kappa:=\kappa_{Q_1}=\kappa_{Q_2}=2^{-\frac{2}{a}}<0.25$ due to the fact $0<a<1$.

\begin{figure}
\centering
\subfigure[]{\includegraphics[width=0.26\textwidth]{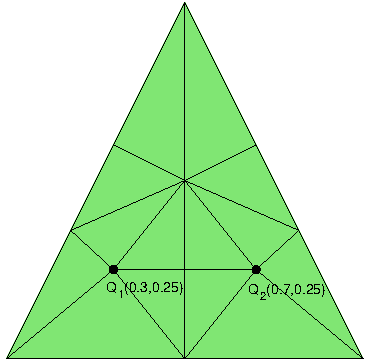}}\hspace{0.8cm}
\subfigure[]{\includegraphics[width=0.26\textwidth]{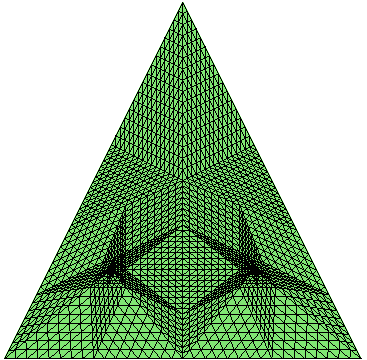}}\hspace{0.8cm}
\subfigure[]{\includegraphics[width=0.31\textwidth]{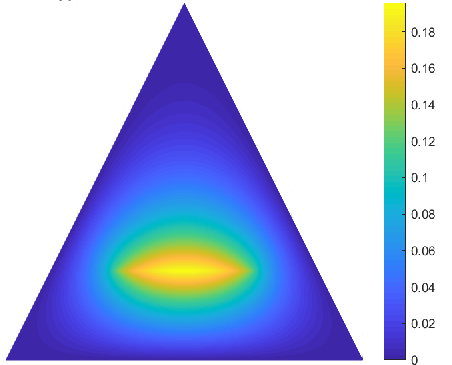}}
\caption{Quadratic finite element methods on graded meshes with the line fracture $\gamma=Q_1Q_2$, $Q_1=(0.3,0.25)$, $Q_2=(0.7,0.25)$. (a) the initial mesh; (b) the mesh after four refinements, $\kappa=\kappa_{Q_1}=\kappa_{Q_2}=0.2$; (c) the numerical solution.}\label{Mesh_InitP2}
\end{figure}

\begin{table}[!htbp]\tabcolsep0.03in
\caption{Convergence history  of the $P_2$ elements in  Example \ref{ex.3} on graded meshes.}
\label{TabConRateP2}
\centering
\begin{tabular}{|l|l|} %{lclclclclcl}
\hline
$\kappa \backslash j$  &  $j=4$    $j=5$    $j=6$    $j=7$  \\
\hline
$\kappa=0.1$   &  1.74  \hspace{0.07cm}  1.86  \hspace{0.07cm}  1.94  \hspace{0.07cm}  1.97  \\
\hline
%$\kappa=0.15$   &  1.80  \hspace{0.07cm}  1.89  \hspace{0.07cm} 1.95  \hspace{0.07cm}  1.98  \\
%\hline
$\kappa=0.2$   &  1.81  \hspace{0.07cm}  1.88  \hspace{0.07cm}  1.93  \hspace{0.07cm}  1.97  \\
\hline
%$\kappa=0.25$   &  1.76  &  1.81  &  1.86  &  1.92  \\
%\hline
$\kappa=0.3$   &  1.65  \hspace{0.07cm} 1.68  \hspace{0.07cm} 1.70  \hspace{0.07cm} 1.71  \\
\hline
$\kappa=0.4$   &  1.32  \hspace{0.07cm}  1.32  \hspace{0.07cm}  1.32  \hspace{0.07cm}  1.32  \\
\hline
$\kappa=0.5$   &  1.00  \hspace{0.07cm}  1.00  \hspace{0.07cm}  1.00  \hspace{0.07cm}  1.00  \\
\hline
\end{tabular}
\end{table}

%Here we solve the problem (\ref{eq:Possion}) by the finite element method on graded mesh given in Figure \ref{Mesh_InitP2}.
The convergence rate (\ref{rate}) of the numerical solution in this example is reported in Table \ref{TabConRateP2}. We observe that the convergence rate is suboptimal on graded meshes with $\kappa>0.25$. In particular, $e=1$ on quasi-uniform meshes ($\kappa=0.5$) and $1<e<2$ on  graded meshes with $\kappa=0.3, 0.4$. It is clear that the optimal convergence rate $e=2$ is obtained on graded meshes when $\kappa<0.25$. These numerical results are clearly consistent with the theory developed in this paper.

\end{example}

%\section{Concluding remarks}\label{conrem}
%In this work, we study the Poisson equation with a line Dirac delta function as a source term subject to Dirichlet boundary condition in a two dimensional domain.
%We studied the regularity of solution in the classical Sobolev space. Due to the lack of the regularity, we further established the regularity of the solution in the weighted Sobolev space by studying the corresponding transmission problem, whose solution was separated into several functions in different spaces. By combing the regularity results in these different spaces, we obtained the regularity of the solution of the transmission problem. Then we established the regularity of the original problem based on the relationship between the transmission problem and the original problem.
%A sequence of graded meshes in the domain is constructed to obtain the quasi-optimal converge rate for the finite element approximation to problem (\ref{eq:Possion}).
%Numerical examples are present to to verify the convergence rate and quasi-optimal convergence rate of $m+1$th order for $P^m$ elements was observed on a sequence of graded triangular meshes with the graded parameter $0<\kappa < 2^{-m/a}$ with $0<a<1$.

\section*{Acknowledgments}
This research was supported in part by the National Science Foundation Grant DMS-1819041 and by the Wayne State University Faculty Competition for Postdoctoral Fellows Award.

\bigskip

\def\cprime{$'$} \def\ocirc#1{\ifmmode\setbox0=\hbox{$#1$}\dimen0=\ht0
  \advance\dimen0 by1pt\rlap{\hbox to\wd0{\hss\raise\dimen0
  \hbox{\hskip.2em$\scriptscriptstyle\circ$}\hss}}#1\else {\accent"17 #1}\fi}

\end{document}